\documentclass{amsart}

\usepackage{amsthm}
\usepackage{amsmath}
\usepackage{color}
\usepackage{tikz}
\usepackage[all]{xy}
\usepackage[margin=1in]{geometry}
\usepackage{verbatim}
\usepackage{mathrsfs}
\usepackage{enumitem}
\usepackage[colorlinks=true,allcolors=purple]{hyperref}
\usepackage{amssymb}
\usepackage{mathtools}

\usetikzlibrary{decorations.markings}
\usetikzlibrary{decorations.pathreplacing}
\usetikzlibrary{positioning}
\usetikzlibrary{arrows,calc,patterns}
\usetikzlibrary{intersections,through}

\usepackage{pgfplots}\pgfplotsset{compat=1.13}

\theoremstyle{plain}
\newtheorem{toptheorem}{Theorem}
\newtheorem{topcorollary}[toptheorem]{Corollary}

\newtheorem{theorem}{Theorem}
\newtheorem{proposition}[theorem]{Proposition}
\newtheorem{corollary}{Corollary}
\numberwithin{corollary}{theorem}
\newtheorem{lemma}{Lemma}

\theoremstyle{definition}
\newtheorem{definition}[theorem]{Definition}
\newtheorem{example}[theorem]{Example}

\theoremstyle{remark}
\newtheorem{warning}[theorem]{Warning}
\newtheorem{remark}[theorem]{Remark}

\numberwithin{theorem}{section}
\numberwithin{equation}{section}
\numberwithin{lemma}{theorem}

\def\overnorm#1{\overline{#1}\vphantom{#1}}
\def\undernorm#1{\underline{#1}\vphantom{#1}}

\mathchardef\mhyphen="2D

\def\Hom{\operatorname{Hom}}
\def\ctM{{}^{\mathrm{ct}}\kern-2pt\mathcal M}
\def\Div{\mathbf{Div}}
\def\ctDiv{{}^{\mathrm{ct}}\Div}
\def\Pic{\mathbf{Pic}}
\def\ctPic{{}^{\mathrm{ct}}\Pic}
\def\Rub{\mathbf{Rub}}
\def\Gm{\mathbf{G}_m}
\def\BGm{{\mathrm{B}\mathbf{G}_m}}
\def\DR{\mathsf{DR}}
\def\Log{\mathbf{Log}}

\def\LogSch{\mathbf{LogSch}}

\def\Mlog{\mathfrak M^{\log}}
\def\can{\rm can}
\def\kcan{k\mhyphen\can}
\def\logGm{\Gm^{\log}}
\def\tropGm{\Gm^{\rm trop}}
\def\Mss{\mathfrak M^{\rm ss}}
\def\stab{\mathrm{stab}}
\def\vir{\mathrm{vir}}

\def\aj{\mathsf{aj}}
\def\Spec{\operatorname{Spec}}

\def\ord{\operatorname{ord}}

\def\NS{\operatorname{NS}}

\begin{document}

\date{\today}
\title{Logarithmic compactification of the Abel--Jacobi section}\

\author{Steffen Marcus}
\address[Marcus]{Mathematics and Statistics\\
The College of New Jersey\\
Ewing, NJ 08628\\
USA}
\email{marcuss@tcnj.edu}

\author{Jonathan Wise}
\address[Wise]{University of Colorado, Boulder\\
Boulder, Colorado 80309-0395\\ USA}
\email{jonathan.wise@math.colorado.edu}

\begin{abstract} 
Given a smooth curve with weighted marked points, the Abel--Jacboi map produces a line bundle on the curve.  This map fails to extend to the full boundary of the moduli space of stable pointed curves.  Using logarithmic and tropical geometry, we describe a modular modification of the moduli space of curves over which the Abel--Jacobi map extends.  We also describe the attendant deformation theory and virtual fundamental class of this moduli space.  This recovers the double ramification cycle, as well as variants associated to differentials.
\end{abstract}


\maketitle

\setcounter{tocdepth}{1}
\tableofcontents

\section{Introduction}

Let $C$ be a smooth curve containing distinct marked points $x_1, \ldots, x_n$.  For any vector $\mathbf a = (a_1, \ldots, a_n)$ of integers, we obtain an invertible sheaf $\mathcal O_C(\sum a_i x_i)$ on $C$.  This construction behaves well in families, and gives a section
\begin{equation*}
\mathcal M_{g,n} \to \Pic_{g,n}
\end{equation*}
where $\Pic_{g,n}$ is the universal Picard group.  We call this map the \emph{Abel--Jacobi section}.

This section fails to extend over $\overline{\mathcal M}_{g,n}$.  As we will recall in Section~\ref{sec:ct}, it does extend to compact type (and even a bit further, though we do not discuss this), but we demonstrate by example in Section~\ref{sec:indeterminacy} that it does not extend near a point corresponding to a curve with two components joined at two nodes.

Our main contribution is the observation that, in compact type, the integer weighting $\mathbf a$ on the marked points of a stable curve $C$ corresponds to an equivalence class of piecewise linear functions on the tropicalization of $C$.  This notion makes sense on the boundary of $\overnorm{\mathcal M}_{g,n}$, and  in Sections~\ref{sec:trop-lines} and \ref{sec:logdiv} we construct a moduli space $\Div_{g,\mathbf a}$ parameterizing stable curves with the choice of such a function.  We show that this additional datum produces a line bundle on $C$ and that this construction resolves the indeterminacy of the Abel--Jacobi section.

\begin{toptheorem} \label{thm:div}
The Abel--Jacobi section extends to a closed embedding $\aj : \Div_{g,\mathbf a} \to \Pic_{g,n}$ over $\overline{\mathcal M}_{g,n}$.
\end{toptheorem}

The Abel--Jacobi \emph{map}, as opposed to the Abel--Jacobi \emph{section}, sometimes refers to the map $\mathcal M_{g,\mathbf a} \to \Pic_g$, that takes in a smooth curve of genus~$g$ with~$n$ distinct, weighted markings and outputs the associated line bundle, \emph{but forgets the location of the markings}.  Unlike the Abel--Jacobi section that is the focus of this paper, this Abel--Jacobi map is usually not an embedding.  Our methods, along with some technical bookkeeping, also resolve the indeterminacy of the Abel--Jacobi map, as we indicate in Section~\ref{sec:aj-map}.  It would be very interesting to understand the fibers of this map, but we make no attempt to do so here.

Theorem~\ref{thm:div} is proven in Section~\ref{sec:logdiv} where it appears as Theorem~\ref{thm:proper}. In Section~\ref{sec:def}, we show that $\Div_{g,\mathbf a}$ is logarithmically \'etale over $\overline{\mathcal M}_{g,n}$, from which it follows that $\Div_{g,\mathbf a}$ has a logarithmic perfect relative obstruction theory and hence a virtual fundamental class:

\begin{toptheorem} \label{thm:vfc}
The Abel--Jacobi section has a logarithmic perfect relative obstruction theory whose obstruction bundle is the dual of the Hodge bundle.
\end{toptheorem}

\subsection*{Rubber targets}

This work grew out of an attempt to understand the moduli space of stable maps from curves to so-called ``rubber'' targets and its virtual fundamental class, and generalizes our earlier work on that problem over the rational tails locus~\cite{CMW} and compact type locus~\cite{MW}.  This space parameterizes maps from prestable curves into unparameterized rational targets.  It was compactified, and equipped with a virtual fundamental class, by Graber and Vakil~\cite{GV}, based on earlier work of Li~\cite{Li1,Li2}, by expanding the rational target in families. 

Section~\ref{sec:rubber} is devoted to comparing our approach to Graber and Vakil's.  We define $\Rub_{g,\mathbf a}$ to be the subfunctor of $\Div_{g,\mathbf a}$ parameterizing those piecewise linear functions on tropicalizations of prestable curves whose values are \emph{totally ordered} in a naturally defined partial order.

\begin{toptheorem}\label{thm:rubbervfc}
The stack $\Rub_{g,\mathbf a}$ is a logarithmic modification of $\Div_{g,\mathbf a}$ and its base change under the trivial section of $\Pic_{g,n}$ over $\overline{\mathcal M}_{g,n}$ is the moduli space of stable maps to rubber targets.  Under this identification, the virtual fundamental classes coincide.
\end{toptheorem}

Both $\Rub_{g,\mathbf a}$ and $\Div_{g,\mathbf a}$ resolve the indeterminacy of the Abel--Jacobi section, but $\Div_{g,\mathbf a}$ is the more efficient of the two, requiring less modification of $\overline{\mathcal M}_{g,n}$.

\subsection*{The double ramification cycle}

One particularly notable application of stable maps to rubber targets has been to complete the double ramification cycle.  As before, let $\mathbf a$ be a set of integer weights whose sum is zero, and let $\DR_{g,\mathbf a}$ be the locus in $\mathcal M_{g,n}$ consisting of those $(C, \mathbf x)$ such that $\mathcal O_C(\sum a_i x_i)$ is isomorphic to $\mathcal O_C$.  In other words, $\DR_{g,\mathbf a}$ is the pullback of the zero locus $\mathcal M_{g,n} \to \Pic_{g,n}$ of $\Pic_{g,n}$ along the Abel--Jacobi section $\aj : \mathcal M_{g,n} \to \Pic_{g,n}$.

The locus $\DR_{g,\mathbf a}$ can also be seen as the collection of marked curves $(C, \mathbf x)$ such that there is a map $f : C \to \mathbf P^1$ with $f^{-1} [\infty] - f^{-1} [0] = \sum a_i x_i$.  Its associated cycle (as a refined intersection) is therefore called the \emph{double ramification cycle}.

Theorem~\ref{thm:div} says that $\aj$ extends to a closed embedding $\Div_{g,\mathbf a} \to \Pic_{g,n}$, and therefore gives a naturally defined extension of $\DR_{g,\mathbf a}$ to a proper morphism:
\begin{equation*}
\Div_{g,\mathbf a}(\mathcal O) = \Div_{g,\mathbf a} \mathop\times_{\Pic_{g,n}} \overnorm{\mathcal M}_{g,n} \to \overnorm{\mathcal M}_{g,n}
\end{equation*}
Theorem~\ref{thm:vfc} gives $\Div_{g,\mathbf a}(\mathcal O)$ a virtual fundamental class in dimension $2g - 3 + n$.

This is not the only way to complete the double ramification cycle, and the original approach relies on the interpretation of $\DR_{g,\mathbf a}$ in terms of stable maps to rational curves.  If $C \to \mathbf P^1$ degenerates so that a component of $C$ maps either to $0$ or to $\infty$, the divisor $f^{-1} [\infty] - f^{-1} [0]$ ceases to have any meaning.  To prevent this from happening, one works in $\Rub_{g,\mathbf a}$ and allows the target $\mathbf P^1$ to expand.  Pulling back $\Rub_{g,n} \mathop\times_{\Pic_{g,n}} \overnorm{\mathcal M}_{g,n}$ yields a second compactification of the double ramification locus, $\Rub_{g,\mathbf a}(\mathcal O)$, and pushing forward its virtual fundamental class gives a second completion of the double ramification cycle.  

\begin{topcorollary}
The two completions $[\Div_{g,\mathbf a}(\mathcal O)]^{\rm vir}$ and $[\Rub_{g,\mathbf a}(\mathcal O)]^{\rm vir}$ of the double ramification cycle coincide.
\end{topcorollary}

Since the two cycle classes coincide, we introduce $[\overnorm\DR_{g,\mathbf a}]$ for either and call it the \emph{(completed) double ramification cycle}.

The double ramification cycle has received considerable attention recently, culminating in the proof of an explicit formula for $[\overnorm\DR_{g,\mathbf a}]$ by Janda, Pandharipande, Pixton, and Zvonkine~\cite{Janda2017} (the formula was previously conjectured by Pixton~\cite{FarkasP}) and attendant tautological relations by Clader and Janda~\cite{CJ}.  It has also led to a search for pluricanonical variants of the double ramification cycle, completing the locus $\DR_{g,\mathbf a}(\omega^{\otimes k})$ consisting of those marked curves $(C,\mathbf x)$ such that $\mathcal O(\sum a_i x_i) \simeq \omega_C^{\otimes k}$.  This question --- which was asked us by F.\ Janda and R.\ Pandharipande --- precipitated the present work.

Theorems~\ref{thm:div} and~\ref{thm:vfc} answer this question.  Let $\Div_{g,\mathbf a}(\omega^{\otimes k})$ be the pullback of the Abel--Jacobi section along the morphism $\overnorm{\mathcal M}_{g,n} \to \Pic_{g,n}$ that sends $(C, \mathbf x)$ to $\omega_C^{\otimes k}$.

\begin{topcorollary}
The projection $\Div_{g,\mathbf a}(\omega^{\otimes k}) \to \overnorm{\mathcal M}_{g,n}$ is proper, and $\Div_{g,\mathbf a}(\omega^{\otimes k})$ has a virtual fundamental class in dimension $2g - 3 + n$.
\end{topcorollary}

The same class can also be constructed using $\Rub_{g,\mathbf a}$.  This facilitates a stable maps interpretation of the pluricanonical double ramification cycle, which we do not pursue here.  Understanding the details of this interpretation would presumably connect our approach to Gu\'er\'e's~\cite{Guere}.

Pixton's conjecture --- but not its proof --- was at least partly inspired by calculations of the double ramification cycle over the compact type locus due to Hain~\cite{H} and Grushevsky--Zakharov~\cite{GZ1} (in a follow up, Grushevsky and Zakharov are able to go slightly beyond compact type~\cite{GZ2}).  Those proofs rely on an extension of the Abel--Jacobi section beyond compact type, and we hope that our work might help to extend the ideas of Hain and Grushevsky--Zakharov to all of $\overline{\mathcal M}_{g,n}$.  For the moment, however, it seems that the lack of a principal polarization on the Picard group of a curve that is not of compact type presents an additional obstacle, and that we will need to work in the logarithmic Picard group instead~(see \cite{logpic1} and a forthcoming
paper by Molcho, Ulirsch and Wise).

\subsection*{Related work}

This paper replaces our preprint~\cite{MW}, which defined closely related moduli spaces but lacked the logarithmic and tropical perspectives employed here, and claimed weaker results.  In the intervening years, several others have achieved similar results, also using logarithmic methods~\cite{Guere,CC1,Holmes}.  We feel that the tropical perspective we employ is natural, and likely to be useful elsewhere (see \cite{genus1}, for example).

Our work can also be seen as a functorial perspective on, and generlization of, the spaces of twisted canonical divisors introduced by Farkas and Pandharipande~\cite{FarkasP} and generalized to $k$-canonical divisors by Schmitt~\cite{Schmitt}.

\subsection*{Logarithmic geometry, tropical geometry, and their relevance to divisors}

We rely heavily on logarithmic geometry as an interpolant between algebraic geometry (narrowly construed) and tropical geometry.  It is a convenient language by which combinatorial data, like dual graphs of algebraic curves, can be permitted to move in algebraic families.  While we summarize some of the main features of logarithmic geometry, as we use it, in Section~\ref{sec:loggeom}, we leave the matter of a thorough introduction to others, such as Ogus's textbook~\cite{Ogus}, K.\ Kato's original papers~\cite{Kato-LogStr,Kato-Toric}, or the surveys~\cite{handbook,skelfan}.  The introduction to logarithmic structures and logarithmic curves in~\cite[Sections~6.1, 6.2, 7.1, and~7.2]{CCUW} is closely aligned with the way we use logarithmic structures here.

Within the logarithmic category, it is possible to speak of morphisms from logarithmic schemes to combinatorial fans.  The details about how this works were the subject of~\cite[Section~6]{CCUW} (see also \cite[Sections~2.2 and~2.3]{LM}).  Perhaps surprisingly, morphisms between these combinatorial data can require meaningful \emph{geometric} input, and therefore be parameterized by geometrically interesting moduli spaces.

Without trying to be too precise yet, a logarithmic structure on a scheme $C$ is a family of line bundles on $C$, indexed by a sheaf of partially ordered abelian groups $\overnorm M_C^{\rm gp}$.  When $C$ is a logarithmic curve over a logarithmic base $S$, we define a logarithmic divisor on $C$ to be a section of this sheaf, $\overnorm M_C^{\rm gp}$ (technically, it will be an equivalence class of such sections).  Our notation for the moduli space of genus~$g$, $n$-marked logarithmic curves equipped with logarithmic divisors is $\Div_{g,n}$.

The dual graph of the logarithmic curve $C$ can be given an integral piecewise linear structure (see \cite[Section~7.2]{CCUW} for the construction) and the sheaf $\overnorm M_C^{\rm gp}$ can be interpreted combinatorially as the sheaf of affine, piecewise linear functions on the dual graph, with respect to this structure (see Section~\ref{sec:pwl} and \cite[Remark~7.3]{CCUW}).  To give such a function requires certain relations to hold among the lengths of the edges of the dual graph, which correspond to certain identities inside $\overnorm M_S^{\rm gp}$, the logarithmic structure of the base.

But the logarithmic structure of $S$ also includes a family of line bundles indexed by $\overnorm M_S^{\rm gp}$, and if one is to impose identities between elements of $\overnorm M_S^{\rm gp}$, one must also have isomorphisms between the corresponding line bundles.  The choices of these isomorphisms are parameterized by tori, which naturally appear as the interiors of the exceptional loci of the birational modification $\Div_{g,n} \to \mathfrak M_{g,n}$ promised in Theorem~\ref{thm:rubbervfc}.

Once $\Div_{g,n}$ has been constructed, the map to $\Pic_{g,n}$ comes for free: a point of $\Div_{g,n}$ is a logarithmic curve, equipped with a section $\alpha$ of $\overnorm M_C^{\rm gp}$; by definition, $\overnorm M_C^{\rm gp}$ parameterizes line bundles on $C$, and sending $(C, \alpha)$ to the corresponding line bundle gives the required map.

\subsection*{Conventions and notation}

For the most part, we work in the logarithmic category, so $T_X$ and $\Omega_X$ denote the logarithmic tangent and cotangent bundles of a logarithmic scheme $X$.  We write $\undernorm X$ for the underlying scheme of a logarithmic scheme, so $T_{\undernorm X}$ and $\Omega_{\undernorm X}$ denote the tangent and cotangent sheaves of the scheme underlying $X$.

The main exception to this convention is for moduli spaces that already have a widely recognized notation, particularly the moduli space of curves.  We use $\mathcal M_{g,n}$ for the moduli space of smooth curves with $n$ distinct markings ($2g -2 + n > 0$), and $\overline{\mathcal M}_{g,n}$ for the Deligne--Mumford--Knudsen compactification.  We write $\mathcal M_{g,n}^{\log}$ for the moduli space of logarithmic curves, whose underlying Deligne--Mumford stack is $\overnorm{\mathcal M}_{g,n}^{\log}$~\cite{FKato_curves}.  We use a fraktur $\mathfrak M$ for pre-stable variants.

\subsection*{Acknowledgements} This work has benefited from helpful and encouraging conversation with Gabriele Mondello and Dhruv Ranganathan.  The questions and comments of Tom Graber, David Holmes, and Rahul Pandharipande have also been very helpful to us.  We owe a particular debt to Felix Janda and Rahul Pandharipande for asking the questions that motivated this project.  We are very grateful as well to the anonymous referees, whose careful reading and suggestions have simplified, corrected, and otherwise improved our work considerably.

J.W.\ was partially supported by NSA grants H98230-14-1-0107 and H98230-16-1-0329 and a Simons Collaboration Grant.

\section{The naive Abel--Jacobi map}

We begin by describing the Abel-Jacobi section over the locus of compact type curves, in order to motivate the main construction of this paper.  This section is not logically necessary for what follows, and may be omitted.  In this section, we do not use any logarithmic structures.

\subsection{Compact type}
\label{sec:ct}
\numberwithin{theorem}{subsection}

\begin{definition}
Let $\ctDiv_g$ be the space of tuples $(C, \mathbf x, \mathbf a)$ consisting of a pre-stable, compact type curve $\pi : C \to S$ of genus $g$, marked sections $x_1, \ldots, x_n$ of $C$ over $S$, and a vector $\mathbf{a} = (a_1, \ldots, a_n)$ of integer weights.  We decorate $\ctDiv$ with subscripts and superscripts to fix combinatorial data:  $\ctDiv^0_g$ is the open and closed substack where $\sum a_i = 0$ and $\ctDiv_{g,\mathbf a}$ is the open and closed substack where the vector $\mathbf a$ has been fixed.
\end{definition}

Note that $\ctDiv_{g,\mathbf a}$ is isomorphic to the stack ${}^{\mathrm{ct}}\mathfrak M_{g,n}$ of $n$-marked, pre-stable, compact type curves.

\begin{definition} \label{def:pic-ct}
We denote by $\ctPic_{g,n}$ the stack whose $S$-points are triples $(C, \mathbf x, L)$ where $(C, \mathbf x)$ is a family of $n$-marked, pre-stable, compact type curves $C$ of genus~$g$, and $L$ is a section of $\pi_\ast(\BGm_C) / \BGm_S$.  We denote by $\ctPic^{0}_{g,n}$ the open substack in which $L$ has degree zero on every component of the fibers of $C$.
\end{definition}

We construct a map $\ctDiv_{g,\mathbf a}^0 \rightarrow \ctPic_{g,n}^0$.  Suppose that $(C, \mathbf x)$ is an $S$-point of $\ctDiv_{g,\mathbf a}$.  We must give a line bundle on $C$, well defined up to tensoring by a line bundle pulled back from $S$.  Since $\ctDiv_{g,\mathbf a}$ is smooth, it is sufficient to treat the case where $S$ is smooth and the family $C$ is a versal deformation of its fibers.  In particular, the total space of $C$ is smooth in this case.

It is also sufficient to work in an \'etale neighborhood of each geometric point of $S$, so we fix one and call it $s$.  We may therefore assume that the boundary of $S$ is a \emph{simple} normal crossings divisor and that all of its components contain $s$.
Each of these components $D_j$ of the boundary of $S$ corresponds to an edge in the dual graph of $C_s$.  Since $C_s$ is of compact type, removing this edge will disconnect the dual graph.  The two components correspond locally in $S$ to divisors in the total space of $C$, so that after further \'etale localization around $s$, we can assume that there are divisors $C_j$ and $C'_j$ such that $\pi^{-1} D_j = C_j + C'_j$.

Now there is a choice of integers $b_j$ and $b'_j$ such that the line bundle $L = \mathcal O_{C}(\sum a_i x_i + \sum b_j C_j + \sum b'_j C'_j)$ has degree zero on every component of $C_s$.  The line bundle $L$ will continue to have degree zero on every component of $C_t$ for every $t$ in an open neighborhood of $s$, so after replacing $S$ by this open neighborhood, we obtain an object of $\ctPic_{g,n}^0$.
Moreover, the construction of $L$ is unique up to replacing $b_j$ and $b'_j$ by $b_j + c$ and $b'_j + c$.  This has the effect of twisting $L$ by $c D_j$, which defines the same object of $\ctPic_{g,n}^0$.  It follows that our construction is independent of our choices and consistent with further localization, and therefore that it glues to a map $\ctDiv_{g,\mathbf a}^0 \rightarrow \ctPic^0_{g,n}$.

There is a twisted variant of this construction taking values in $\ctPic^{\kcan}_{g,n}$, where the superscript $\rm can$ refers to the open and closed substack parameterizing line bundles of the $k$-canonical multidegree ($k(2h-2+m)$ on a component of genus $h$ with $m$ special points).  Let $\mathbf a$ be a tuple of integers such that $\sum a_i = k(2g - 2)$.  Then, as above, there is a unique way of twisting $\mathcal O_C(\sum a_i x_i)$ so that it has the $k$-canonical multidegree on each component.  This gives a map $\ctDiv_{g,\mathbf a}^{\kcan} \to \Pic^{\kcan}_{g,n}$. 

\begin{remark}
We emphasize that our Abel--Jacobi map retains the configuration of the marked points, and is therefore injective on $\ctDiv_{g,\mathbf a}$ for any fixed $\mathbf a$ of total weight zero.  One can, of course, subsequently forget the marked points, which vary in a proper family.
\end{remark}

\begin{proposition}\label{thm:ct-proper}
The map $\ctDiv_{g,\mathbf a}^{\kcan} \to \ctPic_{g,n}^{\kcan}$ is a closed embedding.
\end{proposition} 
\begin{proof}
The projection $\ctPic_{g,n}^{\kcan} \to \ctM_{g,n}$ is separated and the composition $\ctDiv_{g,\mathbf a}^{\kcan} \to \ctM_{g,n}$ is an isomorphism.  Therefore the map $\ctDiv_{g,\mathbf a}^{\kcan} \to \ctPic_{g,n}^{\kcan}$ is a section of a separated map, hence a closed embedding.\footnote{We thank D.\ Holmes and an anonymous referee for pointing out that this simple explanation could replace our original argument.  The same conclusion holds for canonical and pluricanonical variants.}
\end{proof}

It may be important to consider multidegrees other than $0$ and $2g-2$, so we emphasize here that the same proof would show that if $L$ had a different multidegree, there would be a unique twist of $M$ by a Cartier divisor supported in the special fiber that recovers $L$.  It seems difficult to formulate Proposition~\ref{thm:ct-proper} in a global way for other multidegrees using the terminology of this section.  Later on, the more general version will fall out naturally from the formulation in terms of logarithmic geometry.

\begin{definition}
Given a family of compact type curves $C$ over $S$ and a line bundle $L$ of multidegree~$0$ on $C$, let $\ctDiv_g^0(C, L)$ be the locus in $\ctDiv_g^0(C)$ where $\mathcal O_{C}(\sum a_i x_i) \simeq L$.  More generally, if $L$ has $k$-canonical multidegree, define $\ctDiv_g^{\kcan}(C,L)$ to be the locus in $\ctDiv_g^{\kcan}(C,L)$ to be the locus in $\ctDiv_g^{\kcan}(C)$ where $\mathcal O_C(\sum a_i x_i) \simeq L$.
\end{definition}

The following statement follows immediately from the definitions, and is included for the sake of comparison to the logarithmic results that follow.

\begin{proposition}\label{thm:divcartesian}
There are cartesian diagrams:
\begin{equation*} \xymatrix{
\ctDiv_{g, \mathbf a}^0(C,L) \ar[r] \ar[d] & S \ar[d]^{(C,L)} \\
\ctDiv_{g,\mathbf a}^0 \ar[r] & \ctPic_g^0
} \qquad \xymatrix{
\ctDiv_{g, \mathbf a}^{\kcan}(C,L) \ar[r] \ar[d] & S \ar[d]^{(C,L)} \\
\ctDiv_{g,\mathbf a}^{\kcan} \ar[r] & \ctPic_g^{\kcan}
} \end{equation*}
\end{proposition}

In particular, let $S = \ctM_{g,n}$, let $C$ be its universal curve, let $L$ be its structure sheaf, and let $\mathbf a$ be a vector of total weight zero.  Then let $\DR_{g,\mathbf a}$ be the homology class of $\ctDiv_{g,\mathbf a}^0(C,L)$ in $\ctM_{g,n}$ under the horizontal projection (which is a closed embedding).  We will see later that this coincides with the double ramification cycle as defined via stable maps.

\subsection{Indeterminacy of the Abel--Jacobi map}
\label{sec:indeterminacy}

In this section we will use $\mathbf D_{g,\mathbf a}^0$ to denote a \emph{na\"ive} extension of $\ctDiv_{g,\mathbf a}^0$ out of the compact-type locus, which is to say the moduli space of pre-stable marked curves with integral weights summing to zero.  We will see momentarily that this definition is not suitable for defining an Abel--Jacobi map, and we correct it in Section~\ref{sec:log-aj}.

The Abel--Jacobi map becomes indeterminate outside of the compact type locus.  There are two sorts of problems that can occur.  The first, and perhaps more familiar, is when a divisor of total degree zero degenerates so that the degree is not zero on every component, and this defect cannot be corrected because of the topology of the dual graph and the rates of smoothing of the nodes (see Figure~\ref{fig:degen1}).

\begin{figure} 
\begin{center}

\begin{tikzpicture} 
	\draw (2,2.2) .. controls (3,2.5) and (4,2.4) .. (5,2.2) .. controls (6, 2) and (7,1.9) .. (8,2.2);
	\draw (2,5.5) .. controls (3,5.8) and (4,5.7) .. (5,5.5) .. controls (6, 5.3) and (7,5.2) .. (8,5.5);
	\draw (2.5,5.2) to[out=280, in=100] node [circle, fill=black,pos=0.25,inner sep=0pt,minimum size=3pt,label=right:{$p$}] {} node [circle, fill=black,pos=0.75,inner sep=0pt,minimum size=3pt,label=right:{$p'$}] {} (2.5,2.9);
	\draw [name path=one] (6.7,5) to[out=320,in=40] node [circle,fill=black, pos=0.5,inner sep=0pt,minimum size=3pt,label=right:{$p'$}] {} node [pos=1,label=left:{$E$}] {}
(6.7,2.7);
	\draw [name path=two] (7.3,5) to[out=220,in=140] node [circle,fill=black, pos=0.5,inner sep=0pt,minimum size=3pt,label=left:{$p$}] {} node [pos=1,label=right:{$D$}] {} (7.3,2.7);
	\path [name intersections={of=one and two, by={[label=right:{$q$}]Q,[label=right:{$r$}]R}}];
    \node [circle, fill=black,inner sep=0pt,minimum size=3pt] at (Q) {};
    \node [draw=red, circle, thick, inner sep=0pt,minimum size=0.7cm] at (Q) {};
    \node [draw=teal, circle, thick, inner sep=0pt,minimum size=0.7cm] at (R) {};
    \node [circle, fill=black,inner sep=0pt,minimum size=3pt] at (R) {};
    \draw [->] (5,1.8) -- (5,0.5);
	\draw [->] (0,3.5) -- (0,0.5);
	\node at (0,4) {$C$};
	\node at (0,0) {$S$};
	\draw (2,0) -- (7,0) -- (8,0);
	\filldraw (7,0) circle [radius=2pt];
	\node [below] at (7,0) {$t=0$};
    \draw [draw=red,thick] ($(Q)+(0,0.35)$) -- (9.5,5.863);
    \draw [draw=red,thick] ($(Q)-(0,0.35)$) -- (10,4);
    \node [name path=redcircle,draw=red, circle, thick, inner sep=0pt, minimum size=2cm,label=right:{$xy=u$}] at (10,5) {};
    \draw [draw=teal,thick] ($(R)+(0,0.35)$) -- (10,3.5);
    \draw [draw=teal,thick] ($(R)-(0,0.35)$) -- (9.5,1.629);
    \node [name path=tealcircle,draw=teal, circle, thick, inner sep=0pt, minimum size=2cm,label=right:{$zw=v$}] at (10,2.5) {};
    \path [name path=ex1] ($(10,5)-(1,1)$) to ($(10,5)+(1,1)$);
    \path [name path=ex2] ($(10,5)-(-1,1)$) to ($(10,5)+(-1,1)$);
    \path [name path=ex3] ($(10,2.5)-(1,1)$) to ($(10,2.5)+(1,1)$);
    \path [name path=ex4] ($(10,2.5)-(-1,1)$) to ($(10,2.5)+(-1,1)$);
    \node [name intersections={of=ex1 and redcircle, by={red_tr,red_bl}}] {};
    \node [name intersections={of=ex2 and redcircle, by={red_tl,red_br}}] {};
    \node [name intersections={of=ex3 and tealcircle, by={teal_tr,teal_bl}}] {};
    \node [name intersections={of=ex4 and tealcircle, by={teal_tl,teal_br}}] {};
    \draw (red_bl) to (red_tr);
    \draw (red_br) to (red_tl);
    \draw (teal_bl) to (teal_tr);
    \draw (teal_br) to (teal_tl);
    \node [above left] at (red_tl) {$x=0$};
    \node [above right] at (red_tr) {$y=0$};
    \node [below left] at (teal_bl) {$z=0$};
    \node [below right] at (teal_br) {$w=0$};
    \draw [decorate,decoration={brace,amplitude=5pt}] (6.5,5.5) -- (7.5,5.5);
    \node [above] at (7,5.6) {$C_0$};
\end{tikzpicture}
\caption{A family of curves deforming to a singular special fiber $C_0$ consisting of two components meeting at two distinct nodes.}
\label{fig:degen1}
\end{center}
\end{figure} 
If $u$ and $v$ are both unit multiples of the local parameter $t$, then the total space of the family is smooth, and twisting by one of the components will change the degree on each of the two components of the special fiber by an even number.  If we assign odd weights $k$ and $-k$ to $p$ and $p'$ then, since the degree of $\mathcal O_{C_0}(k p - k p')$ is odd, no twist can achieve zero degree on both components.  However, this is really a defect of $\Pic^0$ --- the line bundle with degree zero on the general fiber has no limit on the special fiber that has degree zero on both components --- and is not our concern in this paper.  This defect is corrected by the logarithmic Picard group~\cite{logpic1}.

More serious, at least for our present concerns, is the case where there is a limiting line bundle of degree zero, \emph{but this line bundle depends on the choice of the degenerating family}.  The divisor $k p - k p'$ on the special fiber clearly depends only on the location of $p$ and $p'$ on the special fiber, and not on the family, so the Abel--Jacobi section must become indeterminate at such a point.  Indeed, assuming that $k = 2 \ell$ is even, we can twist by $\ell D$ to get a line bundle of degree zero on the total space.  However, we shall demonstrate below that if $\ell \neq 0$ then twisting by $\ell D$ changes the gluing data at the two nodes in a manner depending on the relative rates of smoothing of the two nodes.
The upshot here is that we cannot have a well-defined map $\mathbf D_{g,\mathbf a}^0 \rightarrow \Pic_{g,n}^0$ on curves of the sort illustrated in Figure~\ref{fig:degen1} unless we modify $\mathbf D_{g,\mathbf a}^0$ to include additional data that can resolve the indeterminacy.

We explain precisely how this indeterminacy arises in the example above (see Figure~\ref{fig:degen1} for notation).  It is interesting to consider a general situation where $u$ and $v$ are not necessarily unit multiples of the smoothing parameter, but where there are nonzero integers $a$ and $b$ such that the ratio of $u^a$ and $v^b$ is a unit, and $a + b = k$.  Of course, $a$ and $b$ must have the same sign if this is to happen.

Under these conditions, we can describe a Cartier divisor $E'$ supported on $E$  by the following local equations:
\begin{equation*} \begin{cases}
x^a  = 0  \quad\text{near $q$} \\
z^b  = 0  \quad\text{near $r$} \\
u^a  = v^b  = 0  \quad\text{near $D$, away from $q$ and $r$} \\
1  = 0 \quad \text{elsewhere}
\end{cases} \end{equation*}
This divisor has numerical intersection $a + b = k$ with $D$.  A similar construction gives a divisor $D'$ having numerical intersection $k$ with $E$.  Then $D' + E'$ is a multiple of a fiber, hence is trivial, and $D'.D = E'.E = -k$.  We note that $\mathcal O_C(D'+E')$ is not canonically trivial, but rather is trivialized by the choice of some $\alpha u^{-a} = \beta v^{-b}$.  We fix one such trivialization.

Now take $L = \mathcal O_C(k p - k p')$.  A line bundle on $C_0$ can be specified by giving a line bundle $L'_D$ on $D$, a line bundle $L'_E$ on $E$, and isomorphisms $L'_D \big|_q \simeq L'_E \big|_q$ and $L'_D \big|_r \simeq L'_E \big|_r$.  We work out these gluing isomorphisms when $L' = L(D')$.  

Set $L_D = L \big|_D$ and $L_E = L \big|_E$.  We certainly have $L_D \big|_q = L_E \big|_q$ and $L_D \big|_r = L_E \big|_r$.  We set $L'_D = L_D(-a q - b r)$ and $L'_E = L_E(a q + b r)$.  The gluing isomorphism is:
\begin{equation*}
L'_D \big|_q = L_D(-a q) \big|_q = L(-E') \big|_q \xrightarrow[\alpha u^{-a}]{\sim} L(D') \big|_q = L_E(a q) \big|_q = L'_E \big|_q
\end{equation*}
A local generator $x^a$ for $L'_D \big|_q$ is carried under this isomorphism to $\alpha u^{-a} x^a = \alpha y^a$ since $xy = u$ near $q$.  Therefore changing the smoothing parameter $u$ at $q$ by a unit $\lambda$ has the effect of scaling the gluing parameter by $\lambda^{-a}$.

A similar calculation near $r$ shows that adjusting $v$ by a unit $\mu$ has the effect of scaling the gluing parameter of $L'$ at $r$ by $\mu^{-b}$.  If $u$ and $v$ are scaled by units so that $u^a / v^b$ is changed, then so will be the limiting line bundle $L'$ in $\Pic^0$ determined by this family.  In particular, if $k \neq 0$, then $L'$ depends on the relative rates of smoothing of the nodes $q$ and $r$.

This calculation also illustrates how to resolve the indeterminacy, at least in this example:  we should blow up $\mathbf D_{g,\mathbf a}$ at the intersection of the boundary divisors corresponding to $Q$ and $R$ so as to introduce a coordinate that keeps track of the ratio between $u^a$ and $v^b$.

Now let us consider a $2$-parameter base, with local coordinates $u$ and $v$.  Then $\mathcal O_C(kp - kp')$ is a section of $\Pic^0(C)$ over the complement of the special point of $S$.  Let $Z$ be its closure.  We will see in Section~\ref{sec:ex-res} that our main theorem implies that $Z$ is not normal if $D$ and $E$ have genus~$0$ and $k > 2$.  As a sanity check, we verify this directly.

We have just seen that the special fiber $Z_0$ of $Z$ is all of $\Pic^0(C_0) \simeq \Gm$.  For any nonzero integers $a$ and $b$ of the same sign such that $a + b = k$, we can find a discrete valuation $\gamma$ of the residue field at the generic point of $S$ such that $\gamma(u^a) = \gamma(v^b)$. Let $T_{a,b}$ be the spectrum of this valuation ring.

By the discussion above, the restriction $Z_{a,b}$ of $Z$ to $T_{a,b}$ has a limit in $Z$, which must be the generic point of $Z_0$.  Each choice of $a$ and $b$ as above therefore gives a distinct valuation of the local ring of $Z$ at the generic point of $Z_0$.  Since $k > 2$, there is more than one choice for $a$ and $b$, which implies that the local ring at $Z_0$ is a $1$-dimensional, noetherian, local ring but not a discrete valuation ring, hence is not integrally closed.

\section{Review of logarithmic geometry} 
\label{sec:loggeom}

\subsection{Logarithmic structures}
\label{sec:logstr}

We briefly recall the needed definitions from logarithmic geometry and refer the reader to \cite{Kato-LogStr} or \cite{Ogus} for a more thorough introduction.

Recall that a monoid is a commutative semigroup with a unit and a homomorphism of monoids is required to preserve the unit element. Given a monoid $M$ we denote by $M^{\rm gp}$ its groupification defined by the universal property that any morphism from $M$ to an abelian group factors through $M^{\rm gp}$ uniquely. A monoid $M$ is called \emph{integral} if the universal map $M\rightarrow M^{\rm gp}$ is injective.  \emph{All monoids in this paper will be integral.}

It is important to note that an integral monoid $M$ can be seen as the set of elements $\geq 0$ in a partial semiorder on the associated group $M^{\rm gp}$.  If the monoid is \emph{sharp}, meaning its only invertible element is the identity, then the partial semiorder is a partial order.  We will use this perspective frequently, particularly in Section~\ref{sec:rub-div}.

A monoid is called \emph{fine} if it is integral and finitely generated. An integral monoid is called \emph{saturated} if $m\in M^{\rm gp}$  and $nm\in M$ for some positive integer $n$ implies $m\in M$. We apply the same adjectives for sheaves of monoids and we often write simply monoid in place of ``sheaf of monoids".

A \emph{logarithmic structure} on a scheme $S$ is an \'etale sheaf of monoids $M_S$ and a homomorphism $\varepsilon : M_S \to \mathcal O_S$ where the target is given its multiplicative monoid structure and such that every unit of $\mathcal O_S$ has a unique preimage under $\varepsilon$.  In other words, $\varepsilon$ restricts to a bijection $\varepsilon^{-1}(\mathcal O_S^\ast) \to \mathcal O_S^\ast$.  We therefore use $\varepsilon^{-1}$ to denote the composition $\mathcal O_S^\ast \to M_S\to M_S^{\rm gp}$. The monoid structure on $M_S$ will be written \emph{multiplicatively}. A morphism $M_S \to N_S$ of logarithmic structures on $S$ is a morphism of sheaves of monoids compatible with their morphisms to $\mathcal O_S$. 

We define $\overnorm M_S = M_S / \varepsilon^{-1}(\mathcal O_S^\ast)$ and call it the \emph{characteristic sheaf of monoids} or simply the \emph{characteristic monoid} of the logarithmic structure.  We notate the monoid structure of $\overline M_S$ \emph{additively}.  By construction, $\overnorm M_S$ is sharp, so the induced semiorder on $\overnorm M_S^{\rm gp}$ is a partial order.  

By definition, there is an exact sequence of \'etale sheaves of abelian groups
\begin{equation*}
0 \to \mathcal O_S^\ast \xrightarrow{\varepsilon^{-1}} M_S^{\rm gp} \to \overnorm M_S^{\rm gp} \to 0
\end{equation*}
which induces a map
\begin{equation*}
\Gamma(S, \overnorm M_S^{\rm gp}) \to H^1(S, \mathcal O_S^\ast) 
\end{equation*}
that will play a crucial role in this paper.  That is, every section $\alpha$ of the characteristic monoid $\overnorm M_S^{\rm gp}$ induces a $\mathcal O_S^\ast$-torsor $\mathcal O_S^\ast(-\alpha)$ on $S$.  In concrete terms, $\mathcal O_S^\ast(-\alpha)$ is the coset of $\varepsilon^{-1} \mathcal O_S^\ast$ in $M_S^{\rm gp}$ classified by $\alpha$.  The restriction of $\varepsilon$ to $\mathcal O_S^\ast(-\alpha)$ induces a homomorphism of line bundles:
\begin{equation*}
\varepsilon_\alpha : \mathcal O_S(-\alpha) \to \mathcal O_S.
\end{equation*}
We also use $\varepsilon_\alpha$ to denote the dual map $\mathcal O_S \to \mathcal O_S(\alpha)$. This construction of the line bundle $\mathcal O_S(\alpha)$ will be used throughout the remainder of the paper.

Logarithmic schemes will be defined to be schemes equipped with logarithmic structures that possess local charts.  We prefer a slightly more permissive definition of a chart, and therefore of quasicoherence, than does Kato~\cite[Section~2.1]{Kato-LogStr}:

\begin{definition} \label{def:chart}
A \emph{chart} for a logarithmic structure $M$ on $X$ is a monoid $P$ and a homomorphism $P \to \Gamma(X, \overnorm M_X)$ such that $M_X$ is the initial example of a logarithmic structure on $X$ admitting such a map.  A logarithmic structure is called \emph{quasicoherent} if it has charts \'etale-locally, and is called \emph{coherent} if it is quasicoherent and the monoids $P$ can be chosen finitely generated.
\end{definition}

\noindent Preferences aside, the difference between Kato's definition of quasicoherence and ours has no logical bearing on this paper.  Our charts are equivalent to Kato's when $P$ is finitely generated, and therefore our coherent logarithmic structures coincide with Kato's, although our quasicoherent logarithmic structures are more inclusive.  Almost all logarithmic structures in this paper will be quasicoherent, and in fact coherent.  The only place where we make use of non-coherent logarithmic structures is in the valuative criterion for properness, proved in Section~\ref{sec:valuative}, but the notion of quasicoherence does not intervene there.

\begin{example}
Suppose that $X$ is a smooth curve of genus $1$ and $A \subset X$ is a set of closed points.  Each finite subset of $A$ generates a divisorial logarithmic structure on $X$.  The union of these logarithmic structures, under the natural inclusion maps, is a quasicoherent logarithmic structure in our sense, but if $A$ is not contained in a finitely generated subgroup of $X$ then this logarithmic structure is not quasicoherent in Kato's sense.
\end{example}

We will use the term \emph{logarithmic scheme} to mean a scheme $S$ equipped with a logarithmic structure $\varepsilon:M_S\to \mathcal O_S$ that is quasicoherent and integral.  If $S$ and $T$ are logarithmic schemes equipped with logarithmic structures and $f : S \to T$ is a morphism of schemes, there is a univeral (initial) way of extending the composition $f^{-1} M_T \to f^{-1} \mathcal O_T \to \mathcal O_S$ to a logarithmic structure $f^\ast M_T$ on $S$.  The characteristic monoid of $f^\ast M_T$ coincides with the \'etale pullback of sheaves $f^{-1} \overnorm M_T$.  A morphism of logarithmic schemes from $(S, M_S)$ to $(T, M_T)$ is a morphism of schemes $f : S \to T$ and a morphism of logarithmic structures $f^\ast M_T \to M_S$.


Morphisms of logarithmic schemes are logarithmically smooth (resp.\ \'etale) if they are locally of finite presentation on the underlying schemes, their logarithmic structures are also locally finitely presented, and they satisfy the right lifting property (resp.\ the unique right lifting property) for strict infinitesimal extensions of logarithmic schemes.  See~\cite[Section~3]{Kato-LogStr} for a thorough introduction to logarithmic differentials and smoothness.

Logarithmic structures will not be assumed saturated, although a saturation hypothesis may usually be imposed with no significant change to the results.  The exception to this rule is Section~\ref{sec:rubber}, where the comparison to Graber and Vakil's moduli space is applies only to the \emph{non-saturated} variants.

\subsection{Logarithmic curves}
\label{sec:logcurves}

We recall the local description of a logarithmic curve. We refer the reader to \cite{FKato_curves} for further background.

\numberwithin{theorem}{subsection}
\begin{definition} \label{def:log-curve}
Let $S$ be a logarithmic scheme.  A \emph{logarithmic curve} over $S$ is a morphism $\pi : C \to S$ of logarithmic schemes that is logarithmically smooth with connected geometric fibers and has the following local description at every geometric point $x$ of $C$ over a geometric point $s$ of $S$: either
\begin{enumerate}[label=(\roman{*})]
\item $x$ is a smooth point of the underlying curve of $C$ and $\overnorm M_{C,x} \simeq \overnorm M_{S,s}$, or
\item $x$ is a smooth point of the underlying curve of $C$ and $\overnorm M_{C,x} \simeq \overnorm M_{S,s} + \mathbf N$, with the additional generator corresponding to a local parameter of $C$ at $x$, or
\item $x$ is a node of the underlying curve of $C$ and $\overnorm M_{C,x} \simeq \overnorm M_{S,s} + \mathbf N \alpha + \mathbf N \beta / (\alpha + \beta = \delta)$ for some $\delta \in \overnorm M_{S,s}$, with $\alpha$ and $\beta$ corresponding to the parameters for $C$ along its branches at $x$ and $\delta$ corresponding to a parameter on the base smoothing the node.
\end{enumerate} 
\noindent In the third part, the element $\delta \in \overnorm M_{S,s}$ is known as the \emph{smoothing parameter} of the node. An alternative local description to (iii) is 
\begin{equation*}
\overnorm M_{C,x} \simeq \{\:(\rho,\sigma)\in \overnorm M_{S,s}\times \overnorm M_{S,s} \:\big|\: \rho-\sigma\in\mathbf Z\delta\:\}.
\end{equation*}
These two descriptions may be identified by sending $\gamma+a\alpha+b\beta$ to $(\gamma+a\delta,\gamma+b\delta)$.
\end{definition}

Logarithmic curves can be defined more parsimoniously as logarithmically smooth, integral, saturated families with connected fibers (see \cite[Definition~1.1 and Theorem~1.1]{FKato_curves} and \cite[Theorem~4.2]{Tsuji}).  The local description will be more important in this paper.

\subsection{Tropical curves}\label{sec:tropcurves}

Suppose that $S$ is a logarithmic point whose underlying scheme is the spectrum of an algebraically closed field and $C$ is a proper logarithmic curve over $S$.  The \emph{tropicalization} of $C$ is the dual graph of $C$, with a leg for each marked point and each edge metrized by its smoothing parameter (Definition~\ref{def:log-curve}). 

\begin{remark} \label{rmk:curve-family}
The smoothing parameters live inside the characteristic monoid $\overnorm M_S$ of the base, which is not necessarily equal to $\mathbf R_{\geq 0}$, so we are not endowing the dual graph with a metric in the usual sense.  However, each homomorphism $\overnorm M_S \to \mathbf R_{\geq 0}$ induces a metric on the dual graph, by applying the homomorphism to the smoothing parameters, so what we have is actually a \emph{family} of metrics, parameterized by $\Hom(\overnorm M_S, \mathbf R_{\geq 0})$.  We therefore ask the reader to indulge us in our abuse of terminology, and permit us to call this labelling a metric.
\end{remark}

\begin{remark}
It is possible to make sense of a varying family of tropical curves over a logarithmic base scheme \cite[Section~7]{CCUW}, but we will not need to do so here.
\end{remark}

Let $\Gamma_C$ be the tropicalization of $C$. Following \cite{CCUW} we understand sections of $\overnorm M^{\rm gp}_C$ through their correspondence to piecewise linear functions on $\Gamma$ valued in $\overnorm M_S^{\rm gp}$ and linear with integer slopes along the edges of $\Gamma$. We summarize this correspondence below, with further details to be found in \cite[Remark~7.3]{CCUW}.

\setcounter{subsubsection}{\value{theorem}}
\subsubsection{Piecewise linear functions} \label{sec:pwl}
Suppose that the underlying scheme of $S$ is the spectrum of an algebraically closed field, that $C$ is a logarithmic curve over $S$, and that $\alpha$ is a section of $\overnorm M_S^{\rm gp}$.  As $\overnorm M_C^{\rm gp}$ is constructible, a section can be described by giving a value in the stalk of $\overnorm M_C$ at a geometric generic point of each stratum, in such that these values are compatible with geometric specialization.  The open strata correspond to the vertices of the dual graph of $C$, and we have $\overnorm M^{\rm gp}_{C,x} = \overnorm M^{\rm gp}_S$ at such a point.  Therefore $\alpha$ induces $\alpha(v) \in \overnorm M^{\rm gp}_S$ for each vertex $v$ of the dual graph of $C$.

These values are not arbitrary, but are constrained by the values at the nodes.  Suppose that $x$ is a node of $C$ generizing to geometric generic points $y$ and $z$.  (If $x$ connected a component of $C$ to itself then $y = z$ but the generization maps are distinct.)  This induces a map
\begin{equation*}
\overnorm M^{\rm gp}_{C,x} \to \overnorm M^{\rm gp}_{C,y} \times \overnorm M^{\rm gp}_{C,z} \simeq \overnorm M^{\rm gp}_S \times \overnorm M^{\rm gp}_S
\end{equation*}
that is injective and furnishes an identification
\begin{equation*}
\overnorm M^{\rm gp}_{C,x} \simeq \{ (a,b) \in \overnorm M^{\rm gp}_S \times \overnorm M^{\rm gp}_S \: \big| \: b - a \in \mathbf Z \delta \}
\end{equation*}
where $\delta \in \overnorm M_S$ is the smoothing parameter of the node $x$.  In the special case where $\overnorm M_S = \mathbf R_{\geq 0}$, we can recognize this characterization of $\overnorm M^{\rm gp}_{C,x}$ as the set of affine linear functions on an interval of length $\delta$ having integer slope.

Applying this observation to every edge of the dual graph of $C$, we see that, if $\overnorm M_S = \mathbf R_{\geq 0}$, a section of $\overnorm M_C^{\rm gp}$ amounts to a piecewise linear function on the dual graph of $C$ that is linear with integer slope along the edges and takes values in $\mathbf R$.  When $\overnorm M_S$ is not $\mathbf R_{\geq 0}$, this interpretation ceases to make literal sense, but we continue to employ the same terminology and regard $\alpha \in \Gamma(C, \overnorm M^{\rm gp}_C)$ as a piecewise linear function on the dual graph of $C$, taking values in $\overnorm M^{\rm gp}_S$.  As in Remark~\ref{rmk:curve-family}, this is not merely an analogy, since each homomorphism $\overnorm M_S \to \mathbf R_{\geq 0}$ induces a piecewise linear function on the graph, as described above, and therefore $\alpha$ may be regarded as a \emph{family} of piecewise linear functions, parameterized by $\Hom(\overnorm M_S, \mathbf R_{\geq 0})$.

\subsubsection{Line bundles on logarithmic curves}

Let $S$ be a logarithmic scheme whose underlying scheme is the spectrum of an algebraically closed field.  As we have seen above, (piecewise linear) maps from a logarithmic curve $C$ over $S$ into $\tropGm$ correspond to piecewise linear functions on the tropicalization, valued in $\overnorm M_S$.  By way of the coboundary map,
\begin{equation*}
H^0(C, \overnorm M_C^{\rm gp}) \to H^1(C, \mathcal O_C^\ast)
\end{equation*}
any such function $\alpha$ induces a line bundle $\mathcal O_C(-\alpha)$ with associated $\mathcal O_C^\ast$-torsor $\mathcal O_C^\ast(-\alpha)$.  Let $C_v$ be the irreducible component of $C$ corresponding to a vertex $v$ of the tropicalization of $C$.  We calculate the restriction, $\mathcal O_{C_v}(-\alpha)$, of $\mathcal O_C(-\alpha)$ to $C_v$.  

Replacing $\alpha$ by $\alpha - \alpha(v)$ will change $\mathcal O_{C_v}(-\alpha)$ by $\mathcal O_{C_v}(\alpha(v))$, which is pulled back from $S$ because $\alpha(v) \in \overnorm M_S^{\rm gp}$.  It will therefore suffice to compute $\mathcal O_{C_v}(-\alpha)$ under the assumption that $\alpha(v) = 0$.

Consider a node or marked point $p$ of $C_v$ along which $\alpha$ has negative slope.  Let $x \in M_{C_v,p}$ be a local coordinate of $C_v$ at $p$ with image $\overnorm x$ in $\overnorm M_{C_v,p}$.  Then $\overnorm M_{C_v,p}^{\rm gp} = \overnorm M_S^{\rm gp} + \mathbf Z \overnorm x$.  As $\alpha(v) = 0$, we have $\alpha_p = \lambda \overnorm x$ for some $\lambda \in \mathbf Z$ --- the slope of $\alpha$ along the edge corresponding to $p$.

By definition, $\mathcal O_{C_v}^\ast(-\alpha)$ is the $\mathcal O_{C_v}^\ast$-torsor of sections of $M_C^{\rm gp}$ projecting to $\alpha$ in $\overnorm M_C^{\rm gp}$.  This means we may identify $\mathcal O_{C_v}^\ast(-\alpha)$ canonically, in a neighborhood of $p$, with $x^{\lambda} \mathcal O_{C_v}^\ast = \mathcal O_{C_v}^\ast(-\lambda p)$.

These observations demonstrate the following proposition:
\begin{proposition} \label{prop:associated-line-bundle}
Let $C$ be a logarithmic curve over $S$, where the underlying scheme of $S$ is the spectrum of an algebraically closed field.  Let $\alpha$ be a piecewise linear function on the tropicalization of $C$.  Then, for each vertex $v$ of the tropicalization of $C$, we have a canonical identification:
\begin{equation*}
\mathcal O_{C_v}(-\alpha) = \mathcal O_S(-\alpha(v)) \otimes \mathcal O_{C_v}(-\sum \lambda_p p)
\end{equation*}
The sum is taken over nodes and marked points of $C$ lying on $C_v$ and $\lambda_p$ is the slope of $\alpha$ along the edge or leg of the dual graph of $C$ corresponding to $p$, oriented away from $v$.
\end{proposition}

\subsection{Logarithmic modifications} \label{sec:log-mod}

For us, a \emph{toric variety} is a normal, equivariant partial compactification of an algebraic torus, with logarithmic structure associated to the inclusion of the dense torus.  Suppose that $\Sigma$ is a rational polyhedral fan in a real vector space with integral lattice $L$.  For any logarithmic scheme $Z$, define $Y_\Sigma(Z)$ to be the set of $\alpha \in \Hom(L, \Gamma(Z, M_Z^{\rm gp})$ such that, locally in $Z$, there is a choice of $\sigma \in \Sigma$ such that $\alpha(\sigma) \subset \Gamma(Z, M_Z)$.  Then $Y_\Sigma$ is representable by a toric variety, and every toric variety arises this way.  If $\Sigma'$ is a subdivision of $\Sigma$ then $Y_{\Sigma'} \to Y_{\Sigma}$ is a birational modification.

\begin{definition}
Let $X$ be a logarithmic scheme.  
A morphism of logarithmic schemes $X' \to X$ that is locally pulled back from a toric modification of toric varieties is called a \emph{logarithmic modification}.
\end{definition}

\begin{proposition} \label{prop:log-mod}
Logarithmic modifications are proper, logarithmically \'etale monomorphisms.  Logarithmic modifications of logarithmically regular\footnote{The definition of logarithmic regularity in \cite[Definition~2.1]{Kato-Toric} requires Zariski-local charts, but is an \'etale-local property by \cite[Theorem~3.1]{Kato-Toric}.  There it is reasonable to apply the term `logarithmically regular' to logarithmic schemes that only have charts \'etale-locally.} logarithmic schemes are also birational.
\end{proposition}
\begin{proof}
In a logarithmically regular logarithmic scheme, the locus where the logarithmic structure is trivial is a dense open subset.  This follows from \cite[Theorem~(3.1)]{Kato-Toric}.  There are no nontrivial logarithmic modifications of a logarithmic scheme with trivial logarithmic structure, so it follows as well that logarithmic modifications of logarithmically regular logarithmic schemes are birational.

The remaining assertions are all local on the base and stable under base change, so it suffices to prove them for an affine toric variety with its toric logarithmic structure.  Suppose $X$ is the affine toric variety corresponding to the rational polyhedral cone $\tau$.  Then the toric modifications $X' \to X$ correspond to the subdivisions of $\tau$ into fans $\Sigma$.  We can describe the functor of points of $X'$ directly in terms of this subdivision:  a map $S \to X'$ is a map $S \to X$ such that $\tau^\vee \to \Gamma(S, \overnorm M_S)$ factors through $\sigma^\vee$ for some $\sigma \in \Sigma$.

We observe that the functor of points of $X'$ is defined by a condition (as opposed to an additional datum) relative to that of $X$, which shows that logarithmic modifications are monomorphisms.  Since this condition can be checked on the level of characteristic monoids, it is automatically logarithmically \'etale.  Properness is well-known from the theory of toric varieties.
\end{proof}

\subsection{The valuative criterion} \label{sec:valuative}

Recall that the pullback of logarithmic structures has a right adjoint~\cite[Section~1.4]{Kato-LogStr}, called pushforward, but that the pushforward of a quasicoherent logarithmic structure is not necessarily quasicoherent.  If $f : X \to Y$ is a morphism of logarithmic schemes and $M_X$ is a logarithmic struture on $X$ then the pushforward logarithmic structure is $\mathcal O_X \mathop\times_{f_\ast \mathcal O_Y} f_\ast M_X$.  

\begin{definition} \label{def:maximal}
Let $j : \eta \to S$ be the inclusion of the generic point in the spectrum of a valuation ring.  If $M_\eta$ is a logarithmic structure on $\eta$, we refer to the pushforward logarithmic structure $\mathcal O_S \mathop\times_{j_\ast \mathcal O_\eta} j_\ast M_\eta$ as the \emph{maximal} extension of the logarithmic structure of $\eta$.
\end{definition}

We again caution that the maximal extension $M_S$ of the logarithmic structure of $\eta$ may fail to be quasicoherent.  This will not be a problem, since we will only be interested in maps \emph{out} of $S$ with its maximal logarithmic structure, and one can make sense of maps into a logarithmic scheme, or moduli problem over logarithmic schemes, from an arbitrary locally ringed space with logarithmic structure.

\begin{theorem}
Let $f : X \to Y$ be a morphism of logarithmic schemes.  The morphism of schemes $\undernorm f : \undernorm X \to \undernorm Y$ underlying $f$ satisfies the valuative criterion for properness if and only if $f$ has the right lifting property with respect to the inclusion of the generic point in the spectrum of a valuation ring with maximal logarithmic structure:
\begin{equation} \label{eqn:23} \vcenter{\xymatrix{
	\eta \ar[r] \ar[d] & X \ar[d]^f \\
	S \ar@{-->}[ur] \ar[r] & Y
} } \end{equation}
The valuative criterion also applies if we assume that the generic logarithmic structure of $S$ is valuative.
\end{theorem}
\begin{proof}
Consider a lifting problem~\eqref{eqn:19}:
\begin{equation} \label{eqn:19} \vcenter{\xymatrix{
	\undernorm \eta \ar[r] \ar[d] & \undernorm X \ar[d] \\
	\undernorm S \ar[r] \ar@{-->}[ur] & \undernorm Y
}} \end{equation}
Let $M_\eta$ be the logarithmic structure on $\undernorm\eta$ pulled back from $X$ and let $M'_S$ be the logarithmic structure on $\undernorm S$ pulled back from $Y$.  Let $M_S$ be the maximal extension of $M_\eta$ to $S$.  Let $N_\eta$ be an extension of the logarithmic structure $M_\eta$, and let $N_S$ be the maximal extension of $N_\eta$ to $S$.  By definition, we have $M_S = N_S \mathop\times_{j_\ast N_\eta} j_\ast M_\eta$.

The commutative square~\eqref{eqn:19} gives a morphism $M'_\eta \to M_\eta$, from which we obtain $M'_S \to M_S$ by adjunction.  We then have a commutative diagram~\eqref{eqn:20} of solid arrows:
\begin{equation} \label{eqn:20} \vcenter{\xymatrix@R=15pt@C=8em{
	(\eta, N_\eta) \ar[r] \ar[d] & (\eta, M_\eta) \ar[r] \ar[d] & X \ar[dd] \\
	(S, N_S) \ar@{-->}[urr]^(.35)h \ar[r] & (S, M_S) \ar@{-->}[ur]_(.45)g \ar[d] \\
	& (S, M'_S)  \ar@{-->}[uur]_(.45){g'} \ar[r] & Y
}} \end{equation}	
Lifting diagram~\eqref{eqn:19} in the category of schemes is equivalent to finding a lifting $h$ in diagram~\eqref{eqn:20}.  We claim that to produce any of the three dashed arrows in diagram~\eqref{eqn:20} is equivalent.  Indeed, $h$ determines $g$ and $g'$ by composition.  But $h$ determines $g$ by the observation that $M_S = N_S \mathop\times_{j_\ast N_\eta} j_\ast M_\eta$.  Likewise, $g$ determines $g'$ by the universal property of the pulled back logarithmic structure $M'_S$.  

Since liftings $g'$ induce liftings $h$, we conclude that if $\undernorm X \to \undernorm Y$ satisfies the valuative criterion for properness then~\eqref{eqn:23} has a lift whenever the logarithmic structure on $S$ is the maximal extension of a logarithmic structure on $\eta$.  Conversely if we have lifting whenever the logarithmic structure is valuative then we can choose $N_S$ above to be valuative (since there is always a map from the logarithmic structure $M_\eta$ to a valuative one) and then we obtain a lift of~\eqref{eqn:19}, proving that $\undernorm X \to \undernorm Y$ satisfies the valuative criterion for properness.
\end{proof}

While it would be possible to formulate a version of this criterion that involves only quasicoherent logarithmic strutures, this will not be necessary for our applications.

\begin{definition}
We say that a homomorphism of integral monoids $f : \overnorm M \to \overnorm N$ is \emph{relatively valuative} if, for every $x \in \overnorm M^{\rm gp}$, if $f(x) = 0$, then either $x \in \overnorm M$ or $-x \in \overnorm M$.
\end{definition}

\begin{remark}
The terminology is motivated from the case where $\overnorm N = 0$, in which the condition says that every $x \in \overnorm M^{\rm gp}$ or its inverse appears in $\overnorm M$.  These are the sorts of monoids that appear as the non-negative elements in the value groups of valuation rings.
\end{remark}

\begin{proposition} \label{prop:rel-val}
Suppose that $S$ is the spectrum of a valuation ring, with the maximal logarithmic structure $M_S$ extending a logarithmic structure at its generic point.  Let $j : \eta \to S$ denote the inclusion of the generic point.  Then  the generization map $\overnorm M_S \to j_\ast \overnorm M_\eta$ is relatively valuative.
\end{proposition}
\begin{proof}
We observe that, by the construction of $M_S = \mathcal O_S \mathop\times_{j_\ast \mathcal O_\eta} j_\ast \overnorm M_\eta$, the kernel of $\overnorm M_S^{\rm gp} \to j_\ast \overnorm M_\eta^{\rm gp}$ is the associated group of the monoid
\begin{equation*}
(\mathcal O_S \mathop\times_{j_\ast \mathcal O_\eta} j_\ast \mathcal O_\eta^\ast) / \mathcal O_S^\ast ,
\end{equation*}
which is simply the valuation monoid of $S$.  It follows immediately that the generization map is relatively valuative.
\end{proof}

\subsection{Logarithmic structures of curves in degenerating families}

Proposition~\ref{prop:ss} is technical, but important in our proof that the Abel--Jacobi map satisfies the valuative criterion for properness (Lemma~\ref{lem:valprop}).  We begin with a local cohomology calculation that will be used in the proof of Proposition~\ref{prop:ss}.  The reader who prefers to do so may omit these preliminaries and rely instead on \cite[Theorem~4.1]{Holmes-Neron} for the same result.

\numberwithin{lemma}{subsection}
\numberwithin{theorem}{subsection}
\numberwithin{equation}{lemma}
\begin{lemma}
Suppose that $A$ is a ring and $t \in A$ is not a zero divisor.  Let $C = \Spec A[x,y] / (xy - t)$, let $i : Z \subset C$ be the inclusion of the vanishing locus of $x$ and $y$, and let $U \subset C$ be its complement.  Then $\mathcal O_C \to j_\ast \mathcal O_U$ is an isomorphism.
\end{lemma}
\begin{proof}
For a quasicoherent sheaf $F$ on $C$ and a closed subset $Z$ of $C$, let $\mathcal H^0_Z(F)$ denote the subsheaf of sections of $F$ with support in $Z$.  We have an exact sequence~\cite[Corollaire~I.2.11]{sga2}:
\begin{equation*}
0 \to \mathcal H^0_Z(\mathcal O_C) \to \mathcal O_C \to \mathcal O_{C-Z} \to \mathcal H^1_Z(\mathcal O_C) \to 0
\end{equation*}
We must show that $\mathcal H^0_Z(\mathcal O_C) = \mathcal H^1_Z(\mathcal O_C) = 0$.  The vanishing of $\mathcal H^0_Z(\mathcal O_C)$ is clear, because $\mathcal O_C$ has no nonzero torsion.  Let $Y$ be the vanishing locus of $y$ in $C$.  As $\mathcal H^0_Z$ has an exact left adjoint, we have a spectral sequence $\mathcal H^p_Z \mathcal H^q_Y (\mathcal O_C) \Rightarrow \mathcal H^{p+q}_Z(\mathcal O_C)$.   It will therefore suffice to show 
\begin{equation}
\mathcal H^0_Z \mathcal H^1_Y (\mathcal O_C) = 0 \qquad \text{and} \qquad \mathcal H^1_Z \mathcal H^0_Y(\mathcal O_C) = 0.
\end{equation}
As $y$ is not a zero divisor in $A[x,y]/ (xy - t)$, we have $\mathcal H^0_Y(\mathcal O_C) = 0$, which gives the latter vanishing above.  We compute $\mathcal H^1_Y(\mathcal O_C)$ with the following exact sequence (using the quasicoherence of the local cohomology sheaves~\cite[Corollaire~II.4]{sga2}):
\begin{equation}
0 \to A[x,y] / (xy - t) \to A[x,y,y^{-1}] / (x - ty^{-1}) \to \mathcal H^1_Y(\mathcal O_C) \to 0
\end{equation}
It is convenient to grade these rings with $x$ having degree $1$ and $y$ having degree $-1$, so that we have the following formula for $\mathcal H^1_Y(\mathcal O_C)$:
\begin{equation*}
\mathcal H^1_Y(\mathcal O_C) = A[y,y^{-1}] / A[ty^{-1},y] = \sum_{n = 1}^\infty  y^{-n} A / y^{-n} t^n A
\end{equation*}
The element $x \in \mathcal O_C$ acts as $ty^{-1}$.  We verify that it is not a zero divisor.  Since $x$ is homogeneous in our grading, it is sufficient to show that $x$ does not act as a zero divisor on any graded piece of $\mathcal H^1_Y(\mathcal O_C)$.  Let $y^{-n} c \in y^{-n} A$ represent an arbitrary element of $y^{-n} A / y^{-n} t^n A$.  If $x c y^{-n} = 0$ then $c t = b t^{n+1}$ for some $b \in A$.  But $t$ is not a zero divisor in $A$, so this implies $c = b t^n$, whence $c y^{-n}$ represents $0$ in $\mathcal H^1_Y(\mathcal O_C)$.  It follows that $x$ is not a zero divisor in $\mathcal H^1_Y(\mathcal O_C)$, so $\mathcal H^0_Z \mathcal H^1_Y(\mathcal O_C) = 0$, as required.
\end{proof}

\setcounter{theorem}{\value{lemma}}
\begin{remark}
In the case where $A$ is noetherian, we could observe that $x$ and $y$ form a regular sequence, and therefore their vanishing locus has depth~$2$ \cite[Proposition~III.3.3]{sga2}.
\end{remark}

\numberwithin{corollary}{subsection}
\setcounter{corollary}{\value{theorem}}
\begin{corollary} \label{cor:local-cohom}
Let $C$ be a nodal curve over $S$ and let $j : U \to C$ be the inclusion of the open subset of $C$ where $C$ is smooth over $S$.  Assume that every node of $Z$ has a local equation $xy - t$ where $t$ is not a zero divisor in $\mathcal O_C$.  Then $\mathcal O_C \to j_\ast \mathcal O_U$ is an isomorphism.
\end{corollary}
\begin{proof}
The assertion is local in the \'etale topology, so we may assume that $S = \Spec A$ and that $C = \Spec A[x,y] / (xy - t)$ where $t$ is not a zero divisor in $A$.  Then the corollary is precisely the assertion of the lemma.
\end{proof}

\numberwithin{theorem}{subsection}
\setcounter{theorem}{\value{corollary}}
\numberwithin{equation}{theorem}
\begin{proposition} \label{prop:ss}
Let $C$ be a logarithmic curve over the spectrum, $S$, of a valuation ring.  Assume that $S$ has the maximal logarithmic structure extending a valuative logarithmic structure on its generic point, $\eta$.  Let $j : C_\eta \to C$ be the inclusion of the generic fiber.  Then the following maps are isomorphisms:
\begin{gather}
M_C \to \mathcal O_C \mathop\times_{j_\ast \mathcal O_{C_\eta}} M_{C_\eta} \label{eqn:25} \\
M_C^{\rm gp} \to j_\ast M_{C_\eta}^{\rm gp} \label{eqn:24}
\end{gather}
In other words, the logarithmic structure of $C$ is the maximal extension of that of $C_\eta$.
\end{proposition}
\begin{proof} 
Our assertion is \'etale-local on $S$, so we may assume $S$ is the spectrum of a \emph{henselian} valuation ring.  

Throughout the proof, we will write $j : \eta \to S$ denote the inclusion of the generic point, as well as the inclusion of the generic fiber $C_\eta \subset C$.  Note that $j_\ast (\mathcal O_\eta^\ast) / \mathcal O_S^\ast$ is the value group of $S$; we denote it $V$.

We write $M_C$ for the logarithmic structure of $C$ and $N_C$ for the pushforward logarithmic structure $\mathcal O_C \mathop\times_{j_\ast \mathcal O_{C_\eta}} j_\ast M_{C_\eta}$ extending $M_{C_\eta}$.  By the universal property of $N_C$, there is a morphism $M_C \to N_C$, which we aim to verify is an isomorphism.  It will be sufficient to check that $\overnorm M_C \to \overnorm N_C$ is an isomorphism.

We check first that~\eqref{eqn:24} is an isomorphism.  We have a commutative diagram with exact rows:
\begin{equation} \label{eqn:15} \vcenter{ \xymatrix{
0 \ar[r] & \mathcal O_C^\ast \ar[r] \ar[d] & M_C^{\rm gp} \ar[r] \ar[d] & \overnorm M_C^{\rm gp} \ar[r] \ar[d] & 0 \\
0 \ar[r] & j_\ast \mathcal O_{C_\eta}^\ast \ar[r] & j_\ast M_{C_\eta}^{\rm gp} \ar[r] & j_\ast \overnorm M_{C_\eta}^{\rm gp} \ar[r] & R^1 j_\ast \mathcal O_{C_\eta}^\ast
}} \end{equation}

Let $\mathscr E$ denote the kernel of $\overnorm M_C^{\rm gp} \to j_\ast \overnorm M_{C_\eta}^{\rm gp}$ and let $\mathscr D$ denote the cokernel of $\mathcal O_C^\ast \to j_\ast \mathcal O_{C_\eta}^\ast$.  We will show that~\eqref{eqn:24} is an isomorphism by showing $\mathcal O_C^\ast \to j_\ast \mathcal O_{C_\eta}^\ast$ is injective (this is immediate), $\overnorm M_C^{\rm gp} \to j_\ast \overnorm M_{C_\eta}^{\rm gp}$ is surjective, and that the connecting homomorphism $\mathscr E \to \mathscr D$ is an isomorphism.

We verify that the $\overnorm M_C^{\rm gp} \to j_\ast \overnorm M_{C_\eta}^{\rm gp}$ is surjective.  Indeed, let $\frak X_\xi \to \frak X_\eta$ be the contraction of dual graphs associated to a specialization $\eta \leadsto \xi$ in $S$.  A local section of $\overnorm M_C^{\rm gp}$ corresponds to function on a subgraph $\Gamma$ of $\frak X_\xi$ that is linear with integer slope along the edges, and takes values in $\overnorm M_\xi^{\rm gp}$.  Let $\Gamma'$ be the image of $\Gamma$ under the edge contraction $\frak X_\xi \to \frak X_\eta$.  A section of $j_\ast \overnorm M_{C_\eta}^{\rm gp}$ corresponds to a function on $\Gamma'$ that is linear with integer slope along the edges and takes values in $\overnorm M_\eta^{\rm gp}$.  Locally, $\Gamma$ is either an edge or a vertex and it is immediate from the surjectivity of $\overnorm M_\eta^{\rm gp} \to \overnorm M_\xi^{\rm gp}$ that linear functions on $\Gamma'$ valued in $\overnorm M_\eta^{\rm gp}$ lift to linear functions on $\Gamma$ valued in $\overnorm M_\xi^{\rm gp}$.

Now we check that $\mathscr E \to \mathscr D$ is an isomorphism.  Sections of $\mathscr E$ correspond to functions on the dual graph of the special fiber of $C$, valued in the kernel of $M_S^{\rm gp} \to M_\eta^{\rm gp}$ and linear along the edges with integral slope.  By definition of the maximal extension $M_S^{\rm gp}$, the kernel of the restriction $M_S^{\rm gp} \to M_\eta^{\rm gp}$ is the value group $V$.  Thus sections of $\mathscr E$ correspond to piecewise linear functions on the dual graph, valued in $V$.

The sections of $\mathscr D$ correspond to invertible sheaves $\mathscr L$ on $C$ with specified trivialization $\sigma$ along $C_\eta$.  Equivalently, $\mathscr D$ is the sheaf of Cartier divisors on $C$ whose support is disjoint from $C_\eta$.  We caution that, because $C$ is not necessarily noetherian, the coefficient of a component $D$ of the special fiber, which is obtained by evaluating $\sigma$ at the generic point $\omega$ of $D$, lies in the value group of the valuation ring $\mathcal O_{\omega}$.  As $C$ is smooth over $S$ at $\omega$, the value group of $\mathcal O_\omega$ is $V$ for all components of the special fiber.  Near a node $Z$ with local equation $xy = t$, the trivialization $\sigma$ must take the form $x^a y^b \sigma_0$, for some integers $a$ and $b$, where $\sigma_0$ is a local trivialization of $\mathscr L$.  Indeed, if $a$ and $b$ are the orders of zero or pole of $\sigma$ along the branches $\{ x = 0 \}$ and $\{ y = 0 \}$ of the node, then $x^{-a} y^{-a} \sigma_0^{-1} \sigma$ is a unit of $\mathcal O_{C-Z}$, hence extends uniquely to a unit of $\mathcal O_C$ by Proposition~\ref{cor:local-cohom}.  If this node connects components $\omega$ and $\psi$ then $\ord \sigma(\omega) = \ord \sigma(\psi) + (b-a)\delta$, where $\delta$ is the image of $t$ in the value group $V$.  Therefore sections of $\mathscr D$ also correspond to functions on the dual graph of the special fiber of $C$ that are valued in $V$ and are linear with integral slope along the edges.

We leave it to the reader to verify that the map $\mathscr E \to \mathscr D$ constructed above commutes with the identifications of $\mathscr E$ and $\mathscr D$ with sheaves of piecewise linear functions and is therefore an isomorphism.  This concludes the demonstration that~\eqref{eqn:24} is an isomorphism.

To prove that~\eqref{eqn:25} is an isomorphism, it is sufficient to show that $\overnorm M_C \to \overnorm N_C$ is an isomorphism.  Note first that we have 
\begin{equation*}
\overnorm M_C^{\rm gp} \to \overnorm N_C^{\rm gp} \subset j_\ast (M_C^{\rm gp}) / \mathcal O_C^\ast
\end{equation*} 
and we have just seen that the composition is an isomorphism.  It follows that $\overnorm N_C^{\rm gp} = \overnorm M_C^{\rm gp}$ and therefore that sections of $\overnorm N_C$ correspond to piecewise linear functions on the dual graph of $C$.  It is immediate that~\eqref{eqn:25} holds away from the nodes of the special fiber of $C$ over $S$, because there we have $\overnorm M_C = \pi^{-1} \overnorm M_S$.  On the other hand, a local section of $\overnorm M_C^{\rm gp}$ lies in $\overnorm M_C$ if and only if its restriction to the complement of the nodes lies in $\overnorm M_C$.  This can be seen from the interpretation in terms of piecewise linear functions on the dual graph, where a section valued in $\overnorm M_S^{\rm gp}$ is $\geq 0$ if and only if its restriction to each vertex is $\geq 0$.  We conclude that $\overnorm N_C$ coincides with $\overnorm M_C$ away from the nodes of the special fiber of $C$, so they agree everywhere.

\end{proof}

\section{The logarithmic compactification of the Abel--Jacobi section} 
\label{sec:log-aj}

\subsection{Tropical line bundles}
\label{sec:trop-lines}

\begin{definition}
For any logarithmic scheme $X$, we define
\begin{align*}
\logGm(X) & = \Gamma(X, M_X^{\rm gp}) \\
\tropGm(X) & = \Gamma(X, \overnorm M_X^{\rm gp})
\end{align*}
and call these, respectively, the \emph{logarithmic multiplicative group} and the \emph{tropical multiplicative group}.  We also call $\tropGm$ the \emph{tropical line}.
\end{definition}

\begin{remark}
As described in Section~\ref{sec:tropcurves}, maps from a curve $C$ over $S$ into $\tropGm$ correspond to piecewise linear functions on the tropicalization, valued in $\overnorm M_S^{\rm gp}$.  Notably, there is no balancing condition, so we use the term `tropical' here with some license.

It seems difficult to characterize balanced functions on the tropicalization sheaf theoretically on $C$.  The image of $\Gamma(C, M_C^{\rm gp}) \to \Gamma(C, \overnorm M_C^{\rm gp})$ consists of some, but in general not all, of the balanced functions on $C$; there is no subsheaf of $\overnorm M_C^{\rm gp}$ whose sections correspond to balanced piecewise linear functions on the tropicalization, because balancing is not a local condition in the \'etale (or Zariski) topology on $C$.  In \cite{logpic1}, balanced functions are characterized as the kernel of the homomorphism $\Gamma(C, \overnorm M_C^{\rm gp}) \to \NS(C)$, where $\NS(C)$ is the N\'eron--Severi group.  This definition seems to generalize well (for example, to surfaces), but it is not a local condition in the Zariski topology on $C$.
\end{remark}

Neither $\logGm$ nor $\tropGm$ is representable by an algebraic stack with a logarithmic structure.  However, $\logGm$ has a logarithmic modification that is representable by $\mathbf P^1$ with its toric logarithmic structure~\cite[Proposition~1]{RW}.  Likewise, $\tropGm$ has a logarithmic modification that is representable by the stack $[\mathbf P^1 / \Gm]$.

\begin{remark} \label{rmk:smoothing}
It may seem odd that a logarithmic modification of the \emph{functor} $\tropGm$ should be representable by the \emph{stack} $[\mathbf P^1 / \Gm]$, but this is a typical phenomenon in logarithmic geometry.  When $[\mathbf P^1 / \Gm]$ is given the right logarithmic structure, it actually does represent a functor and not merely an algebraic stack, because the logarithmic structure serves to rigidify the moduli problem.  Put another way, when $\mathbf P^1$ is given its toric logarithmic structure, $\Gm$ acts on it without fixed points (in a logarithmic sense).  Put still another way, the map from $[\mathbf P^1 / \Gm]$ to Olsson's stack of logarithmic structures~\cite{Olsson-Log} is representable.
\end{remark}

A map $X \to \tropGm$ corresponds to a section $\alpha$ of $\overnorm M_X^{\rm gp}$, and therefore to a $\Gm$-torsor $\mathcal O_X^\ast(\alpha)$ on $X$.  The identity map $\tropGm \to \tropGm$ should therefore induce a universal $\Gm$-torsor on $\tropGm$.  The following lemma characterizes it.

\begin{proposition} \label{prop:univ-torsor}
The universal $\Gm$-torsor over $\tropGm$ is $\logGm$.
\end{proposition}
\begin{proof}
Suppose that $X \to \tropGm$ corresponds to a section $\alpha$ of $\overline M_X^{\rm gp}$.  The associated $\Gm$-torsor is the sheaf of lifts of $\alpha$ to $M_X^{\rm gp}$.  This is precisely the same as the sheaf of lifts of $X \to \tropGm$ to $\logGm$.
\end{proof}

\begin{definition}
Let $S$ be a logarithmic scheme.  A \emph{(piecewise linear) tropical line bundle} over $S$ is a $\tropGm$-torsor on $S$ and a \emph{logarithmic line bundle} is a $\logGm$-torsor over $S$.\footnote{In \cite{logpic1}, the definition of a logarithmic line bundle includes an additional bounded monodromy condition that does not concern us here.}  We will drop the words ``piecewise linear'' in this paper, since we will not encounter any other sort of tropical line bundle here.
\end{definition}

In other words, a tropical line bundle is a torsor under $\overline M_S^{\rm gp}$.  We note that the sections of a tropical line bundle $\mathcal P$ over $S$ are equipped with a partial order, in which $\beta \leq \gamma$ if there is some $\delta \in \Gamma(S, \overline M_S)$ such that $\beta + \delta = \gamma$.

If $P$ is a logarithmic line bundle over $S$ then $P$ carries a natural $\Gm$-action and $P/\Gm$ is a tropical line bundle over $S$.  This is the geometric manifestation of the morphism on first cohomology:
\begin{equation*}
H^1(S, M_S^{\rm gp}) \to H^1(S, \overnorm M_S^{\rm gp})
\end{equation*}




\subsection{Logarithmic divisors on logarithmic curves} \label{sec:logdiv}

\begin{definition} \label{def:logdiv}
Let $C$ be a logarithmic curve over a logarithmic scheme $S$.  A \emph{logarithmic divisor} on $C$ is a tropical line bundle $\mathcal P$ over $S$ and an $S$-morphism $C \to \mathcal P$.  We write $\Div$ for the stack in the strict \'etale topology on logarithmic schemes whose $S$-points are triples $(C, \mathcal P, \alpha)$ where $C$ is a logarithmic curve over $S$, where $\mathcal P$ is a tropical line bundle over $S$, and where $\alpha : C \to \mathcal P$ is an $S$-morphism.
\end{definition}

\begin{example} \label{ex:apwl}
Let $S$ be a logarithmic scheme whose underlying scheme is the spectrum of an algebraically closed field and let $C$ be a logarithmic curve over $S$.  Then $\Div(C/S)$ is the set of piecewise linear functions  on the tropicalization of $C$ that are linear with integer slope along the edges and take values in $\overline M_S^{\rm gp}$ (see Section~\ref{sec:pwl}), up to translation by constants.
\end{example}

The following proposition gives a more concrete formulation of Example~\ref{ex:apwl}:

\begin{proposition}
Let $\pi : C \to S$ be a family of logarithmic curves.  Then $\Div(C/S) = \Gamma\bigl(S, \pi_\ast (\overline M_C^{\rm gp}) / \overline M_S^{\rm gp}\bigr)$.
\end{proposition}
\begin{proof}
Let $\mathcal P$ be a tropical line over $S$.  Locally in $S$ we may choose an isomorphism between $\mathcal P$ and $\tropGm$, which is unique up to addition of a section of $\overline M_S^{\rm gp}$.  Then a map $C \to \mathcal P \simeq \tropGm$ is identified with a global section of $\overline M_C^{\rm gp}$, up to translation by $\overline M_S^{\rm gp}$.
\end{proof}

It is possible to separate $\Div$ into components using several locally constant data.  The first is the genus of the curve in question.  Second, the slopes of the piecewise linear function introduced in Example~\ref{ex:apwl} on the legs of the tropicalization are locally constant in families.  If $\mathbf a = (a_1, \ldots, a_n)$ is a vector of integers, we write $\Div_{g,\mathbf a}$ for the component of the $S_n$ cover of $\Div$ where the curve has genus $g$, the marked points are labeled by the integers $1, \ldots, n$, and the outgoing slope of the piecewise linear function on the tropicalization on the $i$th leg is $a_i$.

\begin{theorem}\label{thm:divisalgebraic}
$\Div$ is representable by an algebraic stack with a logarithmic structure that is locally of finite presentation.
\end{theorem}

\begin{proof}
We employ the standard approach using Gillam's criteria:  first we consider the stack $\Log(\Div_{g,\mathbf a})$ over schemes, whose objects are the same as those of $\Div_{g,\mathbf a}$ but whose arrows are only the cartesian arrows over schemes.  Equivalently, if $S$ is a scheme, then an $S$-point of $\Log(\Div_{g,\mathbf a})$ is the choice of a logarithmic structure $M_S$ on $S$ and an $(S, M_S)$-point of $\Div_{g,\mathbf a}$.  Once we have shown $\Log(\Div_{g,\mathbf a})$ is an algebraic stack, we equip it with the minimal logarithmic structure and identify $\Div_{g,\mathbf a}$ as an open substack.

We begin by showing $\Log(\Div_{g,\mathbf a})$ is algebraic.  Denote by $\Mlog_{g,n}$ the moduli stack of logarithmic curves of genus $g$ with $n$ marked points.  Working relative to $\Log(\Mlog_{g,n})$, we may assume a fixed logarithmic curve $\pi:C\rightarrow S$ has been given. As the formation of $\pi_\ast\overnorm M_C^{\rm gp}$ in the small \'etale topology commutes with base change, so does the formation of $\pi_\ast(\overnorm M_C^{\rm gp})/\overnorm M_S^{\rm gp}$. That is, if $f : T \to S$ is a strict morphism of logarithmic schemes, then $\pi_\ast(\overnorm M_{C_T}^{\rm gp}) / \overnorm M_T^{\rm gp} = f^\ast \bigl( \pi_\ast (\overnorm M_C^{\rm gp}) / \overnorm M_S^{\rm gp} \bigr)$.  Therefore $\Log(\Div_{g,\mathbf a})$ is the pullback to the large \'etale site of the sheaf $\pi_\ast (\overnorm M_C^{\rm gp}) / \overnorm M_S^{\rm gp}$ on the small \'etale site.  It is therefore representable by the espace \'etal\'e \cite[Theorem~V.1.5]{Milne} of $\pi_\ast(\overnorm M_C^{\rm gp})/\overnorm M_S^{\rm gp}$ over $\Log(\Mlog_{g,n})$. This must be locally of finite presentation, since it is \'etale over $\Log(\Mlog_{g,n})$.

Recall that a tuple $(S', C', \alpha')$, where $S'$ is a logarithmic scheme, $C'$ is a logarithmic curve over $S'$, and $\alpha'$ is a logarithmic divisor on $C'$, is called \emph{minimal} if, whenever $(S, C, \alpha)$ is pulled back from $(S', C', \alpha')$ by a morphism $S \to S'$ that is the identity on underlying schemes, and $(S, C, \alpha) \to (S'', C'', \alpha'')$ is a second morphism that is the identity on underlying schemes, there is a unique morphism $(S'', C'', \alpha'') \to (S',C',\alpha')$ completing the triangle below (see \cite{Gillam-Minimality} and \cite[Appendix~B]{minimal} for more details):
\begin{equation} \label{eqn:26} \vcenter{ \xymatrix{
(S,C,\alpha) \ar[dr] \ar[rr] && (S',C',\alpha') \\
& (S'',C'',\alpha'') \ar@{-->}[ur]
} } \end{equation}
By \cite[Proposition~B.1]{minimal}, $\Div_{g,\mathbf a}$ is induced from a logarithmic structure on a moduli problem on logarithmic schemes if and only if every object of $\Div_{g,\mathbf a}$ admits a morphism to a minimal object that is the identity on underlying schemes, and that the pullback of a minimal object under an arbitrary morphism of schemes is still minimal.  In the language of op.\ cit., $G = \Div_{g,\mathbf a}$ is representable by the substack of minimal objects $F \subset \Log(\Div_{g,\mathbf a})$.  Furthermore, \cite[Theorem~B.2]{minimal} implies that, if the minimal objects of $\Div_{g,\mathbf a}$ have coherent logarithmic structures then the substack of minimal objects, $F \subset \Log(\Div_{g,\mathbf a})$, is open.  We may therefore conclude that $\Div_{g,\mathbf a}$ is representable by an algebraic stack with a logarithmic structure once we show that (i) all objects of $\Div_{g,\mathbf a}$ admit morphisms to minimal objects, (ii) minimal objects are preserved by schematic base change, and (iii) minimal logarithmic structures are coherent.

Suppose $\alpha$ is a logarithmic divisor on a logarithmic curve $C$ over a logarithmic scheme $S$, and denote by $N_S$ the \emph{minimal} logarithmic structure on $S$ for the family of curves $C$ (ignoring the logarithmic divisor).  The is the same as the logarithmic structure pulled back from the moduli space of curves; the stalk of $\overnorm N_S$ at a geometric point $s$ is freely generated by the smoothing parameters of the nodes in the fiber $C_s$.  We will show that $\alpha$ is actually defined over the image of $N_S$ in $M_S$.  We work in a small neighborhood of a geometric point $s$ of $S$ so that we may view $\alpha$ as a piecewise linear function on the tropicalization of $C_s$, well-defined up addition of a constant.  Adding a suitable constant to $\alpha$, we can therefore assume that $\alpha$ takes the value $0$ at some vertex of the dual graph of $C_s$.  Since $\alpha$ is linear with integer slopes along the edges, and the lengths of all edges come from the image of $\overnorm N_S$, it follows that all values of $\alpha$ lie in the image of $\overnorm N_S^{\rm gp}$, and therefore that $\alpha$ is defined over the image of $\overnorm N_S$ in $\overnorm M_S$.


Now, let $K\subset \overnorm N_S^{\rm gp}$ be the smallest subgroup such that $\alpha$ can be represented by an element of $\overnorm N_S^{\rm gp}/K$. The subgroup $K$ is generated pointwise as follows: for every directed edge $e$ from a vertex $v$ to a vertex $w$ in the dual graph of $C$, $\alpha$ determines an element $\delta_e = \alpha(w) - \alpha(v) \in \overnorm N_S^{\rm gp}$; the elements of $K$ are the sums $\sum_{e\in\gamma}\delta_e$ for each directed loop $\gamma$ in the dual graph. We will see that the image of $\overnorm N_S$ in $\overnorm N_S^{\rm gp} / K$ is the minimal monoid.  

If we write $(S',C',\alpha')$ for the object of $\Div_{g,\mathbf a}$ whose underlying scheme is the same as that of $S$, whose characteristic monoid $\overnorm M_{S'}$ is the image of $\overnorm N_S$ in $\overnorm N_S^{\rm gp} / K$, and where $\alpha'$ is the canonical lift of $\alpha$ to $\Div_{g,\mathbf a}(C'/S')$ constructed above, then we have a morphism $(S,C,\alpha) \to (S',C',\alpha')$.  To see that $(C',S',\alpha')$ is minimal, consider a diagram~\eqref{eqn:26}.  Since $N_S$ is minimal for the family of logarithmic curves, we have a map $N_S \to M_{S''}$ and, by the same reasoning as above, $\alpha''$ is defined over the image logarithmic structure.  All of the relations in $K$ must hold in $\overnorm M_{S''}$ for the same reasons they held in $\overnorm M_S$, and therefore the map $\overnorm N_S^{\rm gp} / K \to \overnorm N_{S}^{\rm gp}$ factors through $\overnorm M_{S''}^{\rm gp}$.  This completes the diagram~\eqref{eqn:26} and shows that $(S',C',\alpha')$ is minimal.

This construction is compatible with base change, so defines the minimal structures for our logarithmic moduli problem and therefore an algebraic stack.  Therefore $\Div_{g,\mathbf a}$ can be described by a logarithmic structure on a schematic moduli problem.  Since the construction produces coherent logarithmic structures (we have given explicit finitely generated charts), $\Div_{g,\mathbf a}$ is in fact the logarithmic moduli problem associated to an open substack of $\Log(\Div_{g,\mathbf a})$.  This substack is locally of finite presentation because $\Log(\Div_{g,\mathbf a})$ is.
\end{proof}

\begin{proposition}
The forgetful morphism $\Div_{g,\mathbf a}\rightarrow \Mlog_{g,n}$ to the moduli space of logarithmic curves is logarithmically \'etale.
\end{proposition}
\begin{proof}
This is immediate from the fact that $\pi_\ast(\overnorm M_C^{\rm gp})/\overnorm M_S^{\rm gp}$ is an \'etale sheaf.
\end{proof}

\setcounter{corollary}{\value{theorem}}
\begin{corollary}\label{cor:divbirational}
$\Div_{g,\mathbf a}$ is logarithmically smooth and birational to $\Mlog_{g,n}$.
\end{corollary}
\begin{proof}
As $\Div_{g,\mathbf a}$ is logarithmically \'etale over $\Mlog_{g,n}$, and $\Mlog_{g,n}$ is logarithmically smooth (over $\mathbf Z$ with the trivial logarithmic structure), so is $\Div_{g,\mathbf a}$.  It is an isomorphism over the locus where the logarithmic structures are trivial, and in a logarithmically smooth logarithmic scheme, the open subset where the logarithmic structure is trivial is dense.
\end{proof}

\subsection{The Abel--Jacobi section} \label{sec:aj}

\begin{definition} \label{def:pic}
We write $\Pic_{g,n}$ for the stack on logarithmic schemes whose $S$-points are pairs $(C, L)$, where $\pi : C \to S$ is a logarithmic curve over $S$ and $L$ is a section over $S$ of $\pi_\ast(\BGm_C) / \BGm_S = R^1 \pi_\ast \mathcal O_C^\ast$.
\end{definition}

In the absence of logarithmic structures, it is well-known that $\Pic_{g,n}$ is representable by an algebraic stack, provisionally denoted $P$, and that the projection to $\mathfrak M_{g,n}$ is representable by algebraic spaces.  It can be seen easily that the morphism of functors $\Pic_{g,n} \to \mathfrak M_{g,n}$ is strict, and therefore that if $P$ is equipped with the logarithmic structure pulled back from $\mathfrak M_{g,n}$ then $P$ represents $\Pic_{g,n}$.

We construct a canonical map $\aj : \Div_{g,\mathbf a}  \rightarrow \Pic_{g,n}$. Given an object $(\pi:C\rightarrow S,\alpha)$ of $\Div$ over a logarithmic scheme $S$, pushing forward the exact sequence $0\rightarrow \mathcal O^\ast_C\rightarrow M_C^{\rm gp}\rightarrow \overnorm M_C^{\rm gp}\rightarrow 0$ along $\pi$ results in a commutative diagram with exact arrows:
\[ 
\xymatrix{
M^{\rm gp}_S \ar[r] \ar[d] & \overnorm M^{\rm gp}_S \ar[d] \ar[r] & 0  \ar[d] \\
\pi_\ast M^{\rm gp}_C \ar[r] & \pi_\ast \overnorm M^{\rm gp}_C \ar[r] & R^1 \pi_\ast \mathcal O_C^\ast.
} 
\]
\noindent Upon taking quotients in the vertical direction, this yields an exact sequence:
\begin{equation} \label{eqn:exactforsections}
\pi_\ast (M^{\rm gp}_C) / M^{\rm gp}_S \rightarrow \pi_\ast (\overnorm M^{\rm gp}_C) / \overnorm M_S^{\rm gp} \rightarrow R^1 \pi_\ast \mathcal O_C^\ast.
\end{equation}
The second arrow in this exact sequence produces a line bundle class $\aj(\alpha)$ in $\Pic(C)$, inducing our desired morphism.

\begin{theorem} \label{thm:proper}
The morphism $\aj:\Div_{g,\mathbf a} \rightarrow \Pic_{g,n}$ is finite, unramified on the underlying algebraic stacks, and logarithmically a monomorphism.
\end{theorem}

\begin{remark}
The logarithmic structure of $\Div_{g,\mathbf a}$ is not necessarily saturated, so $\aj$ will be finite but not necessarily an embedding if we work with saturated logarithmic structures.
\end{remark}

\begin{proof}
Since both source and target are locally of finite presentation, we may demonstrate that $\aj$ is finite by showing it
\begin{enumerate}
\item is quasicompact,
\item is locally quasifinite, and
\item satisfies the valuative criterion for properness.
\end{enumerate}
Local quasifiniteness follows from being unramified, so the the theorem will follow from the following four lemmas.
\end{proof}

\numberwithin{lemma}{theorem}

\begin{lemma} \label{lem:monomorphism}
The map $\aj : \Div_{g,\mathbf a} \to \Pic_{g,n}$ is a logarithmic monomorphism.
\end{lemma}
\begin{proof}
We learned the following argument from Lionel Levine.  This is a local assertion, so we may assume that we have a logarithmic curve $C$ over $S$ and a geometric point $s$ of $S$ such that an element of $\Div_{g,\mathbf a}(C)$ corresponds to a piecewise linear function on the tropicalization of $C_s$, up to the addition of a constant.  Suppose we have two such piecewise linear functions, $\alpha$ and $\beta$, such that $\aj(\alpha) = \aj(\beta)$.  Then $\aj(\alpha - \beta)$ is pulled back from the base.  We argue that this implies $\alpha - \beta$ is a constant function.  

Indeed, because $\aj(\alpha - \beta)$ is pulled back from the base, it has degree zero on every component of $C$.  The degree of $\aj(\alpha - \beta)$ on a component is the sum of the outgoing slopes of $\alpha - \beta$ from the corresponding vertex.  Therefore the sum of the outgoing slopes of $\alpha - \beta$ is zero at every vertex, so $\alpha - \beta$ is a harmonic function.  On the other hand, it is bounded, because $\alpha$ and $\beta$ have the same outgoing slopes $\mathbf a$  on the legs of the dual graph.  Therefore $\alpha - \beta$ is a bounded, harmonic function on the tropicalization, hence is constant.
\end{proof}

\begin{lemma} \label{lem:quasifinite}
The map $\aj : \Div_{g,\mathbf a} \to \Pic_{g,n}$ is unramified on the underlying algebraic stacks (ignoring logarithmic structures).
\end{lemma}

\begin{lemma}
Let $S$ be the spectrum of an algebraically closed field, let $C$ be a logarithmic curve over $S$, and let $L$ be an invertible sheaf over $C$.  The fiber of $\Div_{g,\mathbf a} \to \Pic_{g,n}$ over $(S,C,L)$ is a discrete set of points.  It is in bijection with the set of assignments of a slope $\lambda_q \in \mathbf Z$ to each edge $e$ of the dual graph of $C$ satisfying the following condition:
  let $D$ be the normalization of all nodes of $C$ where $\lambda_q \neq 0$; then $L \big|_D \simeq \mathcal O_D(\sum \lambda_q q)$, with the sum taken over all edges $q$ of the dual graph of $D$ created by normalization, identified with special points of $D$ and oriented outwards.  In particular, $\aj : \Div_{g,\mathbf a} \to \Pic_{g,n}$ is unramified.
\end{lemma}

\begin{proof}
We thank Sebastian Bozlee for his help with this argument, and David Holmes for alerting us to a missing condition in the statement.

By definition, a $T$-point of the fiber of $\aj$ over $(S,C,L)$ is a morphism of schemes $f : T \to S$, a logarithmic structure $M_T$ on $T$, making $C_T$ into a logarithmic curve over $T$, and a piecewise linear function $\alpha$ on the dual graph of $C$ taking values in $\overnorm M_T^{\rm gp}$, and taken up to the addition of constants, such that $\mathcal O_{C_T}(\alpha) \simeq L \big|_{C_T}$.  


The slopes $\lambda_q \in \mathbf Z$ of $\alpha$ along the edges of the dual graph of $C$ are locally constant on the fiber.  Isolating one connected component, we may treat them as fixed.  We note that these slopes determine the characteristic monoid $\overnorm M_T$ of the logarithmic structure $M_T$, but not the logarithmic structure itself.

We will refer to a union of components $D$, as in the statement of the lemma, as a \emph{level component} of $C$.  Certainly the component of the fiber of $\aj$ with slopes $\lambda_q$ will be empty if there is a level component $D$ of $C$ such that $L \big|_D$ is not isomorphic to $\mathcal O_D(\sum \lambda_q x_q)$, where the sum is taken over edges exiting the vertex of the dual graph corresponding to $D$.  Therefore we may also assume that $\mathcal O_D(\sum \lambda_q x_q) \simeq L \big|_D$ for every level component $D$.

We will now argue that, with $\lambda_q \in \mathbf Z$ fixed as above such that $\mathcal O_D(\sum \lambda_q q)$ for each level component $D$ of $C$, there is a unique choice of minimal logarithmic structure $M_T$ with characteristic monoid $\overnorm M_T$ on $C_T$ such that $\mathcal O_{C_T}(\alpha) \simeq L$.  It will then follow that the component of $\aj^{-1}(S,C,L)$ with these fixed slopes $\lambda_q$ is a single reduced point.

Let $\Gamma$ be the quotient of the dual graph of $C$ in which the vertices correspond to the level components of $C$ and let $\tilde\Gamma$ be its universal cover.  Let $\rho : \tilde C \to C$ be the \'etale cover of $C$ corresponding to the universal cover $\tilde\Gamma \to \Gamma$.  The $\lambda_q$, when pulled back to $\tilde\Gamma$, determine a piecewise linear function $\tilde\alpha$ on $\tilde\Gamma$ taking values in $\overnorm M_S^{\rm gp}$ and having integral slopes along the edges.  This $\tilde\alpha$ is unique up to the addition of a constant from $\overnorm M_S^{\rm gp}$.  

The isomorphisms $L \big|_{D} \simeq \mathcal O_{D}(\sum_{q \in D} \lambda_q q)$ imply that $\rho^\ast L \simeq \mathcal O_{\tilde C}(\tilde \alpha)$.  Since $\tilde C$ is a connected and satisfies the valuative criterion for properness, this isomorphism is unique up to a global scalar.  Since $\rho^\ast L$ descends to $C$, we have identifications $\gamma^\ast \rho^\ast L = \rho^\ast L$ for all $\gamma \in \pi_1(\Gamma)$.  On the other hand, $\gamma^\ast \mathcal O_{\tilde C}(\tilde \alpha) = \mathcal O_{\tilde C}(\gamma^{-1} \tilde\alpha)$ where $\gamma^{-1} \tilde\alpha$ is the piecewise linear function given by the following formula for each level component $\tilde D$ of $\tilde C$:
\begin{equation*}
\gamma^{-1} \tilde\alpha(\tilde D) = \tilde\alpha(\gamma . \tilde D) = \tilde\alpha(\tilde D) + \sum_{q \in \gamma} \lambda_q \delta_q
\end{equation*}
The sum is taken over the oriented edges of $\gamma$ and $\delta_q \in \overnorm M_S$ is the length of the edge $q$.  Thus the isomorphisms $\gamma^\ast \mathcal O_{\tilde C}(\tilde \alpha) \simeq \mathcal O_{\tilde C}(\tilde\alpha)$ are specified uniquely by trivializations of the invertible sheaves
\begin{equation*}
\mathcal O_S(\sum_{q \in \gamma} \lambda_q \delta_q) .
\end{equation*}
If $\gamma$ represents a nontrivial element of $H_1(\Gamma)$ then the $\lambda_q$ are all nonzero for $q \in \gamma$, by the construction of $\Gamma$.  It follows (for example, by the choice of a spanning tree) that the homomorphism 
\begin{equation*}
H_1(\Gamma) \to \overnorm M_S^{\rm gp} : \gamma \mapsto \sum_{q \in \gamma} \lambda_q \delta_q
\end{equation*}
is injective.  Therefore these trivializations are precisely the data necessary to build a quotient logarithmic structure of $M_S$ in which $\sum_{q \in \gamma} \lambda_q \delta_q = 0$ for all $\gamma \in \pi_1(\Gamma)$.  The equations $\sum_{q \in \gamma} \lambda_q \delta_q = 0$ are exactly those required to produce a minimal logarithmic structure on $S$ such that $\tilde\alpha$ descends to a piecewise linear function on $\Gamma$ valued in its characteristic monoid.
\end{proof}

\begin{lemma} \label{lem:valprop}
The map $\aj : \Div_{g,\mathbf a} \to \Pic_{g,n}$ satisfies the valuative criterion for properness.
\end{lemma}
\begin{proof}
Suppose that $C \rightarrow S$ is a family of logarithmic curves, where $S$ is the spectrum of a valuation ring, $j : \eta \to S$ is the inclusion of its generic point, and the logarithmic structure of $S$ is the maximal extension of the logarithmic structure on $\eta$ (as in Definition~\ref{def:maximal}).  That is, $M_S = j_\ast M_\eta \mathop\times_{j_\ast \mathcal O_\eta} \mathcal O_S$.  We consider a lifting problem:
\begin{equation*} \vcenter{\xymatrix{
    \eta \ar[r] \ar[d] & \Div_{g,\mathbf a} \ar[d] \\
    S \ar@{-->}[ur] \ar[r] & \Pic_{g,n}.
}} \end{equation*}
Making a ramified base change, we can assume that $\overnorm M_\eta$ is a constant sheaf on $\eta$.

Let $C_\eta$ be the restriction of $C$ to $\eta$, and let $C_0$ be the special fiber.  The diagram gives a line bundle $L$ on $C$ and a section $\alpha_\eta \in \Gamma(C_\eta, \overnorm M_{C_\eta})$ such that $\aj(\alpha_\eta) = L |_\eta$.  

There is a commutative diagram with exact rows:
\begin{equation} \label{eqn:6} \vcenter{\xymatrix{
0 \ar[r] & M_C^{\rm gp} \ar[r] \ar[d] & \overnorm M_C^{\rm gp} \ar[r] \ar[d] & \BGm_C \ar[d] \\
0 \ar[r] & j_\ast M_{C_\eta}^{\rm gp} \ar[r] & j_\ast \overnorm M_{C_\eta}^{\rm gp} \ar[r] & j_\ast \BGm_{C_\eta}
}} \end{equation}

By Proposition~\ref{prop:ss}, the map $M_C^{\rm gp} \to j_\ast M_{C_\eta}^{\rm gp}$ is an isomorphism, so the square on the right is cartesian.  The line bundle $L$ is a section of $\BGm_C$, the section $\alpha_\eta$ lies in $j_\ast \overnorm M_{C_\eta}^{\rm gp}$, and these have the same image in $j_\ast \BGm_{C_\eta}$, by assumption.  Therefore there is a unique section $\alpha$ of $\overnorm M_C^{\rm gp}$ inducing both $\alpha_\eta$ and $L$.

This does not yet complete the proof, because $\alpha_\eta$ was only well-defined up to addition of a section of $\overnorm M_\eta^{\rm gp}$.  Likewise, $L$ was only well-defined up to tensor product by a section of $j_\ast \BGm_\eta$.  We must show that $\alpha$ does not depend on these choices.  On the other hand, we are only interested in $\alpha$ up to addition of a section of $\overnorm M_S^{\rm gp}$.  We once again have a commutative diagram with exact rows~\eqref{eqn:7}:
\begin{equation} \label{eqn:7} \vcenter{\xymatrix{
0 \ar[r] & M_S^{\rm gp} \ar[r] \ar[d] & \overnorm M_S^{\rm gp} \ar[r] \ar[d] & \BGm_S \ar[d] \\
0 \ar[r] & j_\ast M_{\eta}^{\rm gp} \ar[r] & j_\ast \overnorm M_\eta^{\rm gp} \ar[r] & j_\ast \BGm_\eta
}} \end{equation}
The map $M_S^{\rm gp} \to j_\ast M_\eta^{\rm gp}$ is again an isomorphism (by definition this time, not by Proposition~\ref{prop:ss}) so the square on the right is once again cartesian.

This shows that the ambiguity in the $\alpha$ we have produced is precisely the common ambiguity in the original choices of $L$ and $\alpha_\eta$, and completes the proof.

\end{proof}

\begin{lemma}
The map $\aj : \Div_{g,\mathbf a} \to \Pic_{g,n}$ is quasicompact.
\end{lemma}
\begin{proof}
We must show that, given a quasicompact base $S$, a family of logarithmic curves $C$ over $S$, and a line bundle $L$ on $C$, the fiber product $S \mathop\times_{\Pic_{g,n}} \Div_{g,\mathbf a}$ is quasicompact.  We may freely pass to a constructible and \'etale covers of $S$ and thereby assume that the logarithmic structure of $S$ constant.  We can also assume that $C$ has constant dual graph $G$.  The multidegree of $L$ is then fixed.  

\vskip1ex

\textsc{Bounding the slopes.}
Following \cite{FarkasP}, we argue that there are at most finitely many possibilities for the slopes of a piecewise linear function $\alpha$ on the tropicalization of $C$ such that the multidegree of $\mathcal O_C(\alpha)$ agrees with the multidegree of $L$. 

The degrees of $L$ on each of the irreducible components of $C$ are a tropical divisor --- an integer valued function on the vertices of $G$.  If $\alpha$ is a piecewise linear function on the tropicalization of $C$ then the degree of $\aj(\alpha)$ on a component of $C$ is the deviation from linearity (i.e., the sum of the outgoing slopes) of $\alpha$ at the corresponding vertex of the tropicalization.  In effect, we wish to show that there are finitely many piecewise linear functions $\alpha$ with the same associated tropical divisor.

Any choice of $\alpha$ endows the tropicalization of $C$ with an orientation of its edges (the direction of increase of $\alpha$, treating vertices as equivalent when the slope along an edge connecting them is zero), and since there are finitely many ways to orient a finite graph, we may assume one orientation has been fixed.  This graph can have no directed loops, for if we had a sequence of vertices $v_0, v_1, \ldots, v_n = v_0$ in a directed loop then we would have $\alpha(v_i) - \alpha(v_{i-1}) \in \overnorm M_S$ for all $i = 1, \ldots, n$.  But then $\sum_{i = 1}^n (\alpha(v_i) - \alpha(v_{i-1}) = 0$ implies that all $\alpha(v_i)$ are equal since $\overnorm M_S$ is a sharp monoid.

Now we proceed by induction:  since the graph has no directed loops, there is at least one vertex with only outwardly directed edges.  The degree of $L$ on this component constrains the possible outward slopes of $\alpha$ on the adjacent edges to a finite collection of possibilities.  We therefore assume one is fixed.  Now delete that vertex and the adjacent edges, adjusting the degree of $L$ on the adjacent components to account for the edges that have been removed.  By induction, there are only finitely many possibilities for the slopes.

\vskip1ex

\textsc{Bounding the logarithmic structure.}
We may now assume that the slopes of $\alpha$ on the edges of $G$ are fixed.  We must still show that the space of lifts of $S \to \Pic_{g,n}$ to $\Div_{g,\mathbf a}$ with these slopes is quasicompact.  Any such map will give a quotient $M'_S$ of the logarithmic structure of $M_S$, and the characteristic monoid $\overnorm M'_S$ of this quotient is determined by the fixed slopes of $\alpha$.  Once its characteristic monoid has been determined, the choices of this logarithmic structure vary in a quasicompact family, so we may assume it has been fixed.

\vskip1ex

\textsc{Conclusion.}
Now we have fixed slopes along the edges of $G$ and a logarithmic structure such that these slopes arise from a piecewise linear function.  This gives a uniquely determined logarithmic divisor $\alpha$ on $C$, and we wish to show that the inclusion in $S$ of the locus where $\mathcal O_C(\alpha)$ agrees with $L$ (up to tensor product by a line bundle from the base) is quasicompact.  In fact, it is closed, by the Seesaw Theorem~\cite[Section~5, Corollary~6, p.~54]{AV}.
\end{proof}

The lemmas together complete the proof of Theorem~\ref{thm:proper}. \qed

\vskip1em

We introduce the notation $\Div_{g,\mathbf a}^0$ for the open substack of $\Div_{g,\mathbf a}$ corresponding to $(C, \alpha)$ such that, when $\alpha$ is regarded as an equivalence class of piecewise linear functions on the tropicalization of $C$, it is actually linear.  That is, the outgoing slopes of $\alpha$ at each vertex sum to zero.
We also introduce $\Div_{g,\mathbf a}^{\can}$ for the open substack of $\Div_{g,\mathbf a}$ where the sum of the outgoing slopes at a vertex $v$ of the tropicalization of a curve is $2h - 2 + m$, where $h$ is the valence of $v$ and $m$ is the number of half-edges incident to $v$.

\begin{proposition}
The loci of $(C, \alpha)$ in $\Div_{g,\mathbf a}^0$ and in $\Div_{g,\mathbf a}^{\can}$ such that $C$ is of compact type are equivalent to $\ctDiv_{g,\mathbf a}^0$ and to $\ctDiv_{g,\mathbf a}^{\can}$, respectively.
\end{proposition}
\begin{proof}
On the compact type locus the dual graph of any fiber of $\pi:C\rightarrow S$ is a tree, making the subgroup $K$ of the minimal logarithmic structure in the proof of Theorem~\ref{thm:divisalgebraic} trivial. Thus the minimal logarithmic structures of $\Div_{g,\mathbf a}$ over the compact type locus are the same as those of ${}^{\mathrm{ct}}\Mlog_{g,n}$.  Once the slopes of $\alpha$ on the legs of a tree are specified, there is a unique choice of $\alpha$ that is linear at the vertices (see the proof of Lemma~\ref{lem:monomorphism}) and has the specified slopes on the legs.  Therefore, the only extra datum provided by the section $\alpha$ of $\pi_\ast(\overnorm M_C^{\rm gp})/\overnorm M_S^{\rm gp}$ is the choice of an integer weight $a_i$ at each marked point.

The proof for canonical degrees is similar.
\end{proof}

\subsection{An example of resolved indeterminacy}
\label{sec:ex-res}

We reprise the example of Section~\ref{sec:indeterminacy}.  Let $S_0$ be the spectrum of an algebraically closed field, and $C_0$ the curve shown in the special fiber of Figure~\ref{fig:degen1}.  Let $C$ be a versal deformation of the nodes of $C_0$, with base $S$.  At least after localizing around $S_0$, we can assume that $S$ has two distinct divisors $\Delta_q$ and $\Delta_r$ where the nodes $q$ and $r$, respectively, persist.

Over the complement of $\Delta_q \cap \Delta_r$, the line bundle $\mathcal O_C(np - np')$ has degree zero on the fibers of $C$ and therefore defines a section of $\Pic^0(C)$ over $S$.  Let $Z$ be the closure of this section in $\Pic^0(C)$.

The fiber of $\Pic^0(C)$ over $S_0$ is a $\Gm$-torsor over $\Pic^0(D) \times \Pic^0(E)$, the projection being pullback to the normalization.  The calculations in Section~\ref{sec:indeterminacy} show that, if $n \neq 0$, the fiber of $Z_0$ over a point of $\Pic^0(D) \times \Pic^0(E)$ is either empty or is the entire fiber of $\Pic^0(C)$ over the same point.

\numberwithin{lemma}{subsection}
\begin{lemma} \label{lem:fiber}
If $n \neq 0$, the fiber of $Z_0$ over $(L_D, L_E)$ is nonempty if and only if $L_D \simeq \mathcal O_D(np - aq - br)$ and $L_E \simeq \mathcal O_E(-np' + aq + br)$ where $a$ and $b$ are integers of the same sign such that $a + b = n$.
\end{lemma}
\begin{proof}
There is a symmetry between $n < 0$ and $n > 0$, so we assume without loss of generality that $n > 0$.  

If the fiber of $Z_0$ over $(L_D, L_E)$ is nonempty then we may choose a discrete valuation ring with uniformizer $t$ whose general point lies above the complement of $\Delta_q \cup \Delta_r$ and whose special point lies in $Z_0$.  Let $b$ and $a$ be the orders of tangency to $\Delta_q$ and $\Delta_r$.  Then $L \otimes \mathcal O_C(np' - np)$ restricts to an invertible sheaf associated to a Cartier divisor $\alpha D + \beta E$ supported on the special fiber of $C_T$.  Since $\alpha D + \beta E$ is a Cartier divisor, there must be $k, k', \ell, \ell', c, c' in \mathbf Z$  such that $\alpha = bk + c = ak' + c'$ and $\beta = b\ell + c = a \ell' + c'$ (corresponding to local equations $x^k y^\ell t^c$ for $\alpha D$ and $z^{k'} w^{\ell'} t^{c'}$ for $\beta E$).  These equations imply that $a(k-\ell) = b(\ell' - k')$.  We calculate $L_D$ and $L_E$: 
\begin{align*}
L_D & = \mathcal O_D(np + (\ell - k) q + (\ell' - k') r) \\
L_E & = \mathcal O_D(-np' + (k - \ell) q + (k' - \ell') r)
\end{align*}
Note that $k-\ell$ and $k' - \ell'$ have the same sign, because both signs are the same as the sign of $\alpha - \beta$.  Since $L_D$ and $L_E$ both have degree zero, we have $n = (k-\ell) + (k'-\ell')$.

Conversely, Let $T$ be a curve in $S$ passing through $S_0$, tangent to $\Delta_q$ with order $b$, and tangent to $\Delta_r$ with order $a$.  Let $t$ be a uniformizing parameter for $T$ at $S_0$.  In local coordinates near $q$ on $C_T = C \mathop\times_S T$, we have $xy = t^b$ and near $r$ we have $zw = t^a$.  Then the local equations $x^a=0$ and $z^b=0$ both describe the same Cartier divisor $abD$ on the total space of $C_T$, and $\mathcal O_E(abD) = \mathcal O_E(aq + br)$.  It follows that $\mathcal O_{C_T}(np - np' + abD)$ restricts to $\mathcal O_D(np - aq - br)$ on $D$ and to $\mathcal O_E(-np' + aq + br)$ on $E$.  This shows that $(L_D, L_E)$ is the limit of a curve that lies generically in the complement of $\Delta_q$ and $\Delta_r$ if $L_D$ and $L_E$ have the form required in the statement of the lemma.  
\end{proof}
\numberwithin{lemma}{theorem}

Now we set $\mathbf a = (n, -n)$ and analyze the fiber of $\Div_{g,\mathbf a}$ over $S$.  For simplicity, we focus on the multidegree~$0$ part of $\Div_{g,\mathbf a}$ (that is, the part of $\Div_{g,\mathbf a}$ whose image in $\Pic_{g,n}$ has multidegree~$0$), which we denote $\Div_{g,\mathbf a}^0$.

Let $M_S$ be the minimal logarithmic structure on $S$ associated to the family of curves $C$ over $S$.  Its rank at $S_0$ is $2$.  The fiber of $\Div_{g,\mathbf a}$ over $S_0$ parameterizes quotient logarithmic structures $M'_{S_0}$ of $M_{S_0}$, together with piecewise linear functions $\alpha$ on the dual graph of $C_0$ taking values in $\overnorm M_{S_0}^{\prime\rm gp}$.  We denote this fiber $W$.

Such a piecewise linear function has integral slopes $a$ and $b$ on the edges corresponding to the nodes $q$ and $r$.  These integers must both have the same sign if either is nonzero.  We can compute that
\begin{equation*}
\mathcal O_C(\alpha) \big|_D = \mathcal O_D(np - aq - br)
\end{equation*}
where we have oriented the edges of the dual graph from $E$ to $D$.  The integers $a$ and $b$ are locally constant on $W$, so $W$ has components indexed by partitions of $n$ into nonzero integers $a$ and $b$ of the same sign.  The proof of Lemma~\ref{lem:quasifinite} shows that each of these components maps isomorphically to $Z_0$.

Note that $\Div_{g,\mathbf a}^0(C) \to Z \subset \Pic_{g,2}^0(C)$ is finite and birational onto its image, but in many situations (depending on $g$, $n$, and the locations of $p$, $p'$, $q$, and $r$) it is not an isomorphism.  By Zariski's ``Main Theorem,'' this means that $Z$ is not normal in these cases.  

\subsection{Logarithmic trivializations of line bundles}
\label{sec:line-bundles}

We saw in Theorem~\ref{thm:proper} that $\Div_{g,\mathbf a} \to \Pic_{g,n}$ is a finite map.  Therefore, if $S \to \Pic_{g,n}$ classifies a line bundle $L$ on a logarithmic curve $C$ over $S$, then $S \mathop\times_{\Pic_{g,n}} \Div_{g,\mathbf a}$ is representable by a finite map to $S$ whose fiber over a geometric point $s$ consists of the piecewise linear functions on the tropicalization of $C_s$ whose associated line bundle is $L \big|_{C_s}$.  In this section, we want to give a somewhat more intrinsic characterization of this space.

Suppose that $\mathbf G_m$ acts on $Y$ and that $L$ is a line bundle with associated $\mathbf G_m$-torsor $L^\ast$; that is, $L^\ast$ is the sheaf of isomorphisms $\mathcal O \to L$ and $L = (L^\ast \times \mathcal O) / \Gm$, where the action is antidiagonal.  We denote by $L^\ast \otimes Y = (L^\ast \times Y) / \Gm$ the contraction of $L^\ast$ with $Y$.  Thus $L = L^\ast \otimes \mathcal O$.

Recall from Section~\ref{sec:logstr} that, given $\alpha\in\pi_\ast (\overnorm M_C^{\rm gp})$ on a logarithmic curve $\pi:C\rightarrow S$ there is a canonical line bundle $\mathcal O_C(-\alpha)$ on $C$ associated to $\alpha$ given by taking the line bundle associated to the $\Gm$-torsor $\mathcal O_C^\ast(-\alpha)$ produced by the preimage of $\alpha$ through the quotient map $M_C^{\rm gp} \rightarrow \overnorm M_C^{\rm gp}$.  If $\alpha$ is instead a logarithmic divisor in $\pi_\ast( \overnorm M_C^{\rm gp}) / \overnorm M_S^{\rm gp}$ then by choosing a representative in $\pi_\ast\overnorm M_C^{\rm gp}$, we obtain a line bundle that is well defined up to tensor product by $\mathcal O_C(\delta)$ for $\delta \in \overnorm M_S^{\rm gp}$.

\begin{definition} \label{def:logtriv}
Let $\pi : C \rightarrow S$ be a logarithmic curve and let $L$ be a line bundle on $C$.  A \emph{logarithmic trivialization} of $L$ over $S$ is a section of the quotient $\pi_\ast (L^\ast \otimes M_C^{\rm gp})/M_S^{\rm gp}$.
\end{definition}



Let $C \to S$ be a fixed family of $n$-marked logarithmic curves over $S$, fix an equivalence class of line bundles $L$ on $C$, up to tensor product by line bundles from the base, and fix a vector $\mathbf a$ of integer weights.

\begin{definition}
We denote by $\Div_{g,\mathbf a}(C,L)$ the category whose objects are:
\begin{enumerate}
\item a logarithmic scheme $S'$ and map $f : S' \to S$, with $C' = f^{-1} C$, and
\item a logarithmic trivialization of $f^\ast L$.
\end{enumerate}
This naturally forms a category fibered in groupoids over $\LogSch$.
\end{definition}

\begin{proposition}\label{prop:divdiagram}
There is a cartesian diagram:
\begin{equation*} \xymatrix{
\Div_{g,\mathbf a}(C, L) \ar[r] \ar[d] & S \ar[d]^{(C, L)} \\
\Div_{g,\mathbf a} \ar[r]_{\aj} & \Pic_{g,n}.
} \end{equation*}
In particular, $\Div_{g,\mathbf a}(C,L)$ is proper over $S$.
\end{proposition}
\begin{proof}
Note that $L^\ast \otimes M_C^{\rm gp}$ is an $\mathcal O_C^\ast$-torsor over $\overnorm M_C^{\rm gp}$.  We regard this as an exact sequence of pointed stacks:%
\begin{equation*}
0 \to L^\ast \otimes M_C^{\rm gp} \to \overnorm M_C^{\rm gp} \to \mathrm B \mathcal O_C^\ast
\end{equation*}
Pushing forward by $\pi_\ast$ gives an exact sequence of pointed sets~\eqref{eqn:exactforsections2}, as in~\eqref{eqn:exactforsections}:
\begin{equation} \label{eqn:exactforsections2}
0 \to \pi_\ast (L\otimes M^{\rm gp}_C)  \rightarrow \pi_\ast \overnorm M^{\rm gp}_C  \to \pi_\ast \mathrm B \mathcal O_C^\ast 
\end{equation}
Dividing by the sequence
\begin{equation*}
0 \to M_S^{\rm gp} \to \overnorm M_S^{\rm gp} \to \mathrm B \mathcal O_S^\ast \to 0
\end{equation*}
and noting that $\pi_\ast (\mathrm B \mathcal O_C^\ast) / B \mathcal O_S^\ast = R^1 \pi_\ast \mathcal O_C^\ast$, 
we obtain another exact sequence
\begin{equation} \label{eqn:exactforsections3}
0 \to \pi_\ast(L \otimes M_C^{\rm gp}) / M_S^{\rm gp} \to \pi_\ast(\overnorm M_C^{\rm gp}) / \overnorm M_S^{\rm gp} \to R^1 \pi_\ast \mathcal O_C^\ast
\end{equation}
Here, the map from $\pi_\ast (\overnorm M_C^{\rm gp}) / \overnorm M_S^{\rm gp} \to R^1\pi_\ast \mathcal O_C^\ast$ sends the equivalence class of $\alpha \in \pi_\ast \overnorm M_C^{\rm gp}$ to $L(-\alpha)$.  The sequence therefore identifies $\pi_\ast(L \otimes M_C^{\rm gp}) / M_S^{\rm gp}$ --- that is, $\Div_{g,\mathbf a}(C,L)$ --- with the the $\alpha \in \Div_{g,\mathbf a}(C)$ whose image in $\Pic_{g,n}$ is $L$.  This is the content of the cartesian diagram in the statement of the proposition.
\end{proof}

This proposition can be applied with $S = \Mlog_{g,n}$, with $C$ taken to be the universal curve, and with $L = \omega_{C/S}^k$.  Then $\Div_{g,n}(C, \omega_{C/S}^k)$ is a compactification of the moduli space of $k$-differentials on curves, and we will see in Section~\ref{sec:def} that it is equipped with a virtual fundamental class. In compact type, $\aj$ is the usual Abel--Jacobi map and the virtual fundamental class of $\Div(X,L)$ restricts to the refined intersection of the Abel--Jacobi section and the section determined by $L$.

\numberwithin{equation}{subsection}
\subsection{Twisted double ramification cycle}
\label{sec:def}

The composition of the sequence of morphisms~\eqref{eqn:sequence}
\begin{equation} \label{eqn:sequence}
\Div_{g,\mathbf a} \xrightarrow{\aj} \Pic_{g,n} \xrightarrow{\Pi} \Mlog_{g,n}
\end{equation}
is logarithmically \'etale, the map $\aj$ is a closed embedding, and the map $\Pi$ is strict (by definition of the logarithmic structure on $\Pic_{g,n}$) and smooth.  This implies that $\Pic_{g,n} \mathop\times_{\Mlog_{g,n}} \Div_{g,\mathbf a}$ is strict and smooth over $\Div_{g,\mathbf a}$, and the induced section $\sigma$ is therefore a local complete intersection embedding.  Writing $I$ for the sheaf of ideals of $\sigma$, we therefore have an isomorphism
\begin{equation*}
\aj^\ast I = \aj^\ast \Omega_{\Pic_{g,n} / \Mlog_{g,n}}.
\end{equation*}
Furthermore $\aj^\ast I[1]$ can serve as the obstruction bundle of a relative obstruction theory for $\sigma$ (we may identify $\aj^\ast I[1]$ with the relative cotangent complex of $\sigma$).  Since 
\begin{equation*}
\Pic_{g,n} \mathop\times_{\mathclap{\Mlog_{g,n}}} \Div_{g,\mathbf a} \to \Pic_{g,n}
\end{equation*}
is the base change of a logarithmically \'etale morphism, it is also logarithmically \'etale over $\Pic_{g,n}$.  Therefore our obstruction theory also works as a relative \emph{logarithmic} obstruction theory for $\aj : \Div_{g,\mathbf a} \to \Pic_{g,n}$.  Indeed, the logarithmic cotangent complex of $\Div_{g,\mathbf a} \to \Pic_{g,n} \mathop\times_{\frak M_{g,n}^{\log}} \Div_{g,\mathbf a} \to \Pic_{g,n}$ is quasi-isomorphic to the relative cotangent complex of $\Div_{g,\mathbf a} \to  \Pic_{g,n}$ by the transitivity triangle \cite[1.1~(v) and Remark~3.14]{Olsson-LCC} (note that logarithmically \'etale morphisms are logarithmically flat).

It is also well-known that the relative tangent bundle of $\Pic_{g,n}$ over $\Mlog_{g,n}$ has fiber $H^1(C, \mathcal O_C)$ over the point corresponding to a logarithmic curve $C$.  This proves Theorem~\ref{thm:vfc}.

It follows that, for any $f : S \to \Pic_{g,n}$ classifying a logarithmic curve $C$ over $S$ and line bundle $L$ on $C$, the pullback $\Div(C,L)$ of $\aj : \Div_{g,\mathbf a} \to \Pic_{g,n}$ has a perfect relative obstruction theory $f^\ast \aj^\ast I$.
Let $S$ be an algebraic stack with a logarithmic structure, let $C$ be a logarithmic curve of genus $g$ over $S$, and let $L$ be a line bundle on $C$.  These data give a morphism $S \to \Pic_{g,n}$.
Fix an $n$-tuple of integers $\mathbf a$.  We have (Section~\ref{sec:line-bundles}) 
\begin{equation*}
\Div_{g,\mathbf a}(C,L) = S \mathop\times_{\hskip1em\mathclap{\Pic_{g,n}}} \Div_{g,\mathbf a} .
\end{equation*}
Gysin pullback for the Abel--Jacobi section equips the map $\Div_{g,\mathbf a}(C,L) \to S$ with a virtual pullback.  In particular, if $S$ has a fundamental class $[S]$ (or a virtual fundamental class) then we obtain a relative virtual fundamental class $[\Div_{g,\mathbf a}(C,L)/S]^{\vir}$ by applying the virtual pullback to $[S]$.

We can apply this in particular where $S = \overline{\mathcal M}_{g,n}$, where  $C$ is the universal curve, and where $L$ is $\omega_{C/S}^k$ for any integer $k$.  This gives a cycle class
\begin{equation*}
\bigl[\Div_{g,\mathbf a}(\mathcal M_{g,n}^{\log}, \omega^k)\bigr]^{\vir}
\end{equation*}
in $\mathrm{CH}_{2g - 3 + n}\bigl(\Div_{g,\mathbf a}(\mathcal M_{g,n}^{\log}, \omega^k)\bigr)$.  Since $\Div_{g,\mathbf a}(\omega^{\otimes k})$ is proper over $\overline{\mathcal M}_{g,n}$, we may push this cycle class forward to $\overline{\mathcal M}_{g,n}$ to obtain a twisted double ramification cycle $\DR_{g,\mathbf a}(\mathcal M_{g,n}^{\log}, \omega^k)$.

\subsection{The Abel--Jacobi map}
\label{sec:aj-map}

Here we construct a proper map $\Div_{g,\mathbf a} \to \Pic_g$ extending the map $\mathcal M_{g,\mathbf a} \to \Pic_g$ that sends a smooth curve $C$ with $n$ distinct marked points $x_1, \ldots, x_n$ to $\mathcal O_C(\sum a_i x_i)$.

Suppose that $C$ is a logarithmic curve without markings.  By a \emph{punctured model} of $C$ we will mean a map $\tau : \tilde C \to C$ such that, if the markings of $\tilde C$ are forgotten, the stabilization of $\tau$ is an isomorphism.  We call it a \emph{stable punctured model} if the underlying curve $\tilde C$ is stable.  If $\mathfrak C$ is a tropical curve and $\alpha$ is a piecewise linear function on $\mathfrak C$ that is linear on the edges, we call $\alpha$ \emph{linear} at a vertex of $\mathfrak C$ if the sum of the outgoing slopes of $\alpha$ at that vertex is zero.  We apply the same terminology to logarithmic divisors.

Let $\Div_{g,\mathbf a}^\ast$ be the stack over logarithmic schemes whose $S$-points are the following data:
\begin{enumerate}
    \item a logarithmic curve $C$ over $S$ without marked points,
    \item a stable punctured model $\tau : \tilde C \to C$ with $n$ marked points, and
    \item a logarithmic divisor $\alpha : \tilde C \to \mathcal P$ and slopes $\mathbf a$ on its legs, such that
    \item $\alpha$ is linear at the components of $\tilde C$ collapsed by $\tau$.
\end{enumerate}

Let us also denote by $\Pic_{g,n}^\ast$ the stack over logarithmic schemes whose $S$-points are
\begin{enumerate}
    \item a logarithmic curve $C$ over $S$ without marked points,
    \item a stable punctured model $\tau : \tilde C \to C$ with $n$ marked points, and
    \item an equivalence class of line bundles $\tilde L$ on $\tilde C$, up to tensor product by line bundles from the base, such that
    \item $L$ has degree zero on all components collapsed by $\tau$.
\end{enumerate}

If $\alpha : \tilde C \to \mathcal P$ is an object of $\Div_{g,\mathbf a}^\ast$ then $\mathcal O_{\tilde C}(\alpha)$ is an object of $\Pic_{g,n}^\ast$.  By Theorem~\ref{thm:proper}, we have a proper map:
\begin{equation*}
\Div_{g,\mathbf a}^\ast \to \Pic_{g,n}^\ast.
\end{equation*}
On the other hand, if $(\tau : \tilde C \to C, \tilde L)$ is an object of $\Pic_{g,n}^\ast$ then $\tilde L$ is the pullback to $\tilde C$ of the line bundle $\tau_\ast \tilde L$ on $C$.  Therefore $\Pic_{g,n}^\ast$ is equivalent to the space of
\begin{enumerate}
    \item a logarithmic curve $C$ over $S$ without marked points,
    \item a stable punctured model $\tau : \tilde C \to C$ with $n$ marked points, and
    \item an equivalence class of line bundles $L$ on $C$.
\end{enumerate}
That is, $\Pic_{g,n}^\ast \to \Pic_g$ is proper, with the preimage of $(C,L)$ corresponding to space of punctured models of $C$ (equivalently, the space of stable curves that stabilize to $C$ when their marked points are forgotten).

\section{Rubber comes from the tropics} \label{sec:rubber}

\subsection{Aligned logarithmic curves}
\label{sec:aligned}

\begin{definition} \label{def:aligned}
Let $C$ be a logarithmic curve over $S$ and let $\alpha : C \to \mathcal P$ be a logarithmic divisor on $C$.  We say that $\alpha$ is \emph{aligned} if, for each geometric point $s$ of $S$, the set of $\alpha(v)$, as $v$ ranges among the vertices of $C_s$, is totally ordered.  We write $\Rub$ for the subfunctor of $\Div$ whose objects are aligned logarithmic divisors.
\end{definition}

\begin{proposition}
$\Rub$ is a logarithmic modification of $\Div$, and in particular is representable by an algebraic stack with a logarithmic structure.
\end{proposition}
\begin{proof}
It is sufficient to show that $\Rub \mathop\times_\Div S$ is a logarithmic modification for each logarithmic scheme $S$ and map $S \to \Div$.  In other words, we must show that if $S$ is a logarithmic scheme, $C$ is a logarithmic curve over $S$, and $\alpha : C \to \mathcal P$ is a logarithmic divisor on $C$, then the space of alignments of $\alpha$ is a logarithmic modification of $S$.

This is a local assertion in the strict \'etale topology on $S$, so we can assume that there is a geometric point $s$ of $S$ such that
\begin{equation*}
\Gamma(S,\overline M_S) \to \overline M_{S,s}
\end{equation*}
is a bijection and that the family of dual graphs of the fibers of $C$ over the closed stratum of $S$ has trivial monodromy.

Let $\Gamma$ be the tropicalization of $C_s$.  For each vertex $v$ of $\Gamma$, we get a value $\alpha(v) \in \overline M_S^{\rm gp}$, and $\alpha$ is aligned if and only if these elements are totally ordered.  

Passing to an \'etale cover of $S$, we can assume that all of the $\alpha(v)$ are pulled back along some map from $S$ to a toric variety $V$.  The space of alignments of $\alpha$ is also pulled back from $V$, where it is easily seen to be representable by a toric modification:  subdivide the cone of $V$ along the hyperplanes $\alpha(v) = \alpha(w)$.
\end{proof}

As a special case of the constructions in this section and as motivation for what's to come, we identify the fiber product
\begin{equation*}
\Rub_{g,\mathbf a}(\mathcal M_{g,n}^{\log}, \mathcal O) = \Rub_{g,\mathbf a} \:\mathop\times_{\mathclap{\hskip1em\Pic_{g,n}}} \:\mathcal M^{\log}_{g,n}
\end{equation*}
where $\mathcal M^{\log}_{g,n}$ is the moduli space of stable logarithmic curves and the map $\mathcal M^{\log}_{g,n} \to \Pic_{g,n}$ sends a curve to the trivial line bundle on that curve. The following theorem will be demonstrated in Section~\ref{sec:rubber-maps}.

\begin{theorem}
$\Rub(\mathcal M_{g,n}^{\log}, \mathcal O)$ is the moduli space of predeformable maps from prestable curves to rubber targets.  The stable locus is the space of stable maps to rubber targets introduced by Graber and Vakil~\cite{GV}.
\end{theorem}

\subsection{Divided tropical lines}
\label{sec:divided}

\begin{definition}
Let $S$ be a logarithmic scheme.  A \emph{divided tropical line bundle} over $S$ is a tropical line bundle $\mathcal P$ and a subfunctor $\mathcal Q \subset \mathcal P$ such that, locally in $S$, there are sections $\gamma_0 \leq \cdots \leq \gamma_n$ of $\mathcal P$ and $\mathcal Q$ is the set of sections of $\mathcal P$ that are comparable to all of the $\gamma_i$.
Given any tropical line $\mathcal P$, we write $\mathcal P_\gamma$ for the divided tropical line over $S$ associated to a sequence $\gamma_0 \leq \cdots \leq \gamma_n$ of $\mathcal P$.  If $\mathcal P$ is the tropical line bundle associated to a logarithmic line bundle $P$ then we also write $P_\gamma = \mathcal P_\gamma \mathop\times_{\mathcal P} P$.
\end{definition}

\begin{example}
Over any base $S$, we may take $P = \logGm$, $\mathcal P = \tropGm$, and $\gamma_0 = 0$ to get a divided tropical line $\mathcal P_\gamma$ and a divided logarithmic line $P_\gamma$.  In this case $P_\gamma = \mathbf P^1$, with its standard logarithmic structure, and $\mathcal P_\gamma = [ \mathbf P^1 / \Gm ]$. 
\end{example}

If $\mathcal P_\gamma$ is a divided tropical line over $S$, the sequence $\gamma_0 \leq \cdots \leq \gamma_n$ cannot be recovered from $\mathcal P_\gamma$.  However, we will show that the differences $\delta_i = \gamma_{i+1} - \gamma_i$ can (see Figure~\ref{fig:pretrop-line}).

Consider the locus $U_i \subset \mathcal P_\gamma$ consisting of those $\alpha$ such that $\gamma_i \leq \alpha \leq \gamma_{i+1}$.  This is an open subfunctor of $\mathcal P_\gamma$.  Let $\beta = \alpha - \gamma_i$ and let $\beta' = \gamma_{i+1} - \alpha$.  Then $\beta + \beta' = \delta_i$, which means that 
\begin{equation*}
U_i(T) \simeq \{ (\beta, \beta') \in \Gamma(T, \overline M_T) \times \Gamma(T, \overline M_T) \: \big| \: \beta + \beta' = \delta_i \}
\end{equation*}
functorially in $T$.  It follows from this that the minimal logarithmic structure of $S$ over which $\mathcal P_\gamma$ is defined is freely generated by the nonzero $\delta_i$.  (If $\delta_i = 0$, then $U_i$ is an open subset of $U_{i+1}$ and $U_{i-1}$, so it can be omitted from discussion.)

We argue that the order of the $\delta_i$ is also determined by $\mathcal P_\gamma$.  Indeed, $U_i \cap U_j \neq \varnothing$ if and only if $\gamma_{i+1} = \gamma_j$ or $\gamma_{j+1} = \gamma_i$.  It is not possible for both to occur unless $\gamma_i = \gamma_{i+1} = \gamma_j = \gamma_{j+1}$, in which case $\delta_i = \delta_j = 0$.  Otherwise, we can recognize that $i \leq j$ if $\gamma_{i+1} = \gamma_j$ and that $j \leq i$ if $\gamma_{i} = \gamma_{j+1}$. 

We also observe that if $\mathcal Q$ is a divided tropical line with lengths $\delta_1, \ldots, \delta_n$ and we set
\begin{gather*}
\gamma_i = \sum_{j \leq i} \delta_j \\
\gamma'_i = -\sum_{j \geq i} \delta_j
\end{gather*}
then $\mathcal Q \simeq \mathcal P_\gamma \simeq \mathcal P_{\gamma'}$ where $\mathcal P = \tropGm$.  In other words, the tropical line bundle underlying a divided tropical line bundle has two canonical trivializations, corresponding to the smallest and largest vertices of the subdivision.

\begin{remark}
Minimal tropical lines with at least one division were called \emph{aligned logarithmic structures} in \cite{ACFW}, where they were studied in terms of the sequence $\delta_1, \ldots, \delta_n$ and the condition~\eqref{eqn:3}:
\begin{equation} \label{eqn:3}
0 \leq \delta_1 \leq \delta_1 + \delta_2 \leq \cdots \leq \delta_1 + \cdots + \delta_n.
\end{equation}
\end{remark}

\begin{figure}
\begin{tikzpicture}[arrow/.style={decoration={markings, mark=at position .5 with {\arrow{angle 90}}},postaction={decorate}}]
\coordinate (A) at (-3,0);
\coordinate (B) at (-1,0);
\coordinate (C) at (1,0);
\coordinate (D) at (3,0);
\coordinate (E) at (5,0);
\draw[arrow] (A) -- (B);
\draw[arrow] (B) -- (C);
\draw[arrow] (C) -- (D);
\draw[arrow] (D) -- (E);
\draw[black,fill=black] (B) circle (.6mm);
\draw[black,fill=black] (C) circle (.6mm);
\draw[black,fill=black] (D) circle (.6mm);
\node at (0,0.4) {$\delta_1$};
\node at (2,.4) {$\delta_2$};
\node at (-2,.4) {$\infty$};
\node at (4,.4) {$\infty$};
\end{tikzpicture}
\caption{A divided tropical line with $\gamma_2 = \gamma_1 + \delta_1$ and $\gamma_3 = \gamma_2 + \delta_2$.}
\label{fig:pretrop-line}
\end{figure}

\begin{proposition} \label{prop:chain}
Let $\mathcal P$ be a tropical line.  If $\gamma = \{ \gamma_1 \leq \cdots \leq \gamma_n \}$ is a nonempty collection of sections of $\mathcal P$ then $\mathcal P_\gamma$ is representable by an algebraic stack over $S$ with a logarithmic structure and $P_\gamma$ is representable by a family of $2$-marked, semistable, genus~$0$ curves.
\end{proposition}
\begin{proof}
The assertion is local in $S$, so we may assume $\mathcal P \simeq \tropGm$.


We have $P_\gamma = \mathcal P_\gamma \mathop\times_{\tropGm} \logGm$.  We will check that $P_\gamma$ is a family of $2$-marked semistable curves over $S$.  The assertion about $\mathcal P_\gamma$ follows, since $\mathcal P_\gamma = [P_\gamma / \Gm]$.

The assertions of the proposition are local in $S$, so we may assume that the $\gamma_i$ lift to elements $\tilde\gamma_i \in \Gamma(S, M_S^{\rm gp})$.  Then a morphism $X \to P_\gamma$ consists of a morphism $f : X \to S$ and $\tilde\alpha \in \Gamma(X, M_X^{\rm gp})$ with image $\alpha \in \Gamma(X, \overnorm M_X^{\rm gp})$ such that $\alpha$ is comparable to every $\gamma_i$.  

We argue first that $P_\gamma$ is a flat family of nodal curves.  Consider the locus $U_i \subset P_\gamma$ where $\gamma_i \leq \alpha \leq \gamma_{i+1}$.  Taking $\tilde\delta_i = \tilde\gamma_{i+1} \tilde\gamma_i^{-1}$, we can identify $U_i(X)$ with the locus of pairs $(\tilde\beta, \tilde\beta') \in \Gamma(X, M_X^{\rm gp})$ such that $\tilde\beta \tilde\beta' = \tilde\delta_i$.  In other words, there is a cartesian diagram of logarithmic schemes:
\begin{equation*} \xymatrix@C=6em{
    U_i \ar[r]^{(\tilde\beta,\tilde\beta')} \ar[d] & \mathbf A^2 \ar[d] \\
    S \ar[r]^{\tilde\delta_i} & \mathbf A^1
} \end{equation*}
where the map $\mathbf A^2 \to \mathbf A^1$ is multiplication.  This proves $U_i$ is representable by a logarithmic scheme, flat over $S$, and has nodal curves as its fibers.

At the extremes, $i = 0$ and $i = n$, a similar argument gives $U_i \simeq \mathbf A^1_S$.

Now we argue that $P_\gamma$ is proper.  For this we may assume that $S$ is the spectrum of a valuation ring with generic point $\eta$, that $S$ has the maximal logarithmic structure extending the logarithmic structure from $\eta$, and that we are given a section of $P_\gamma$ over $\eta$.  That is, we are given an element $\tilde\alpha$ of $\Gamma(\eta, M_\eta^{\rm gp})$ that is comparable to the $\tilde\gamma_i$.  We wish to show that this extends uniquely to $\Gamma(S, M_S^{\rm gp})$.  But by construction, $\Gamma(\eta, M_\eta^{\rm gp}) = \Gamma(S, M_S^{\rm gp})$, so we obtain the lift of $\tilde\alpha$ automatically.  The comparability to the $\tilde\gamma_i$ is not as immediate, because the partial order on $M_S^{\rm gp}$ has fewer relations than does $M_\eta^{\rm gp}$.  However, because $M_S$ is the maximal extension of $M_\eta$, any two elements of $\overnorm M_S^{\rm gp}$ whose images are comparable in $\overnorm M_\eta^{\rm gp}$ are also comparable in $\overnorm M_S$ (Proposition~\ref{prop:rel-val}).  This completes the proof of properness.

Finally, we check that the fibers of $P_\gamma$ are semistable, $2$-marked curves.  We may assume $S$ is the spectrum of an algebraically closed field.  Let $V_i$ be the locus in $P_\gamma$ where $\alpha = \gamma_i$.  Then we combine
\begin{equation*}
P_\gamma = U_0 \mathop\cup_{V_1} U_1 \mathop\cup_{V_2} U_2 \mathop\cup_{V_3} \cdots \mathop\cup_{V_{n-1}} U_{n-1} \mathop\cup_{V_n} U_n
\end{equation*}
with the identifications 
\begin{align*}
U_0 & \simeq \mathbf A^1_S \\
U_i & \simeq S \mathop\times_{\mathbf A^2} \mathbf A^1 \\
U_n & \simeq \mathbf A^1
\end{align*}
and a quick check of the gluing to confirm that $P_\gamma$ is a $2$-marked, semistable curve.
\end{proof}

The proposition shows that any tropical line $\mathcal P_\gamma$ over $S$ with at least one division gives rise to a family of $2$-marked, semistable, rational curves $P_\gamma$.  Thus $\gamma$ determines a morphism $S \to \Mss_{0,2}$.  Conversely, if $P$ is such a family of curves over a logarithmic scheme $S$, then locally $P$ has smoothing parameters $\delta_1, \delta_2, \ldots, \delta_n \in \Gamma(S, \overnorm M_S)$ associated to its nodes.  The sequence of elements
\begin{equation*}
\gamma_i = \sum_{j \leq i} \delta_i
\end{equation*}
is then a divided tropical line over $S$.

These two operations are not quite inverse to one another because the family of rational curves carries a $\Gm$ automorphism that is not visible in the tropical line.  It is shown that in \cite{ACFW} that when this automorphism is rigidified away, $\Mss_{0,2}$ is indeed the moduli space of tropical lines with at least one division.

\numberwithin{equation}{theorem}
\begin{theorem} \label{thm:T}
    There is an algebraic stack $\mathscr T$ with logarithmic structure parameterizing families of tropical lines with at least one division, and $\Mss_{0,2} \simeq \mathscr T \times \BGm$.
\end{theorem}

We sketch the argument using the terminology of this paper.  Suppose that $P$ is a family of $2$-marked, semistable, genus~$0$ curves over $S$.  If $s$ is a geometric point of $S$ and $P_s$ has $n$ nodes then $\overline M_{S,s}$ contains smoothing parameters for those nodes $\delta_1, \ldots, \delta_n$; because $P_s$ is a chain of rational curves, the $\delta_i$ have a canonical order, with the indices corresponding to proximity to the first marked point.  As we have seen above, the $\delta_i$ give a divided tropical line over $S$.

We can also extract a line bundle from $P$ by taking the normal bundle at the first marked point.  Together, these give a map:
\begin{equation} \label{eqn:4}
\Mss_{0,2} \to \mathscr T \times \BGm.
\end{equation}

Conversely, if $\mathcal P_\gamma$ is a divided tropical line over $S$ then $P_\gamma$ is a family of $2$-marked, genus~$0$ semistable curves by Proposition~\ref{prop:chain}.  Twisting by $L$, the family $L \otimes P_\gamma$ is also a chain of rational curves.  The map that sends $(\mathcal P_\gamma, L)$ to $L \otimes P_\gamma$ is inverse to~\eqref{eqn:4}.

We note that we could have used the normal bundle at the second marked point in the discussion above and obtained a second identification between $\Mss_{0,2}$ and $\mathscr T \times \BGm$.  These identifications can be related by a formula involving the $\delta_i$.

\subsection{Rubber divisors}
\label{sec:rub-div}

Suppose that $C$ is a logarithmic curve over $S$ and that $\alpha : C \to \mathcal P$ is a logarithmic divisor on $C$.  For each irreducible component $v$ of each geometric fiber of $C$ over $S$, the restriction of $\alpha$ to $v$ determines an element $\alpha(v) \in \overnorm M_S^{\rm gp}$, well-defined up to simultaneous shift of all $\alpha(v)$.  As $\overnorm M_S^{\rm gp}$ is partially ordered (with $\overnorm M_S$ being the cone of elements $\geq 0$), this gives a partial order on the vertices of the dual graph of each fiber of $C$ over $S$.  This partial order is compatible with generization in the evident sense: whenever $v$ and $w$ are irreducible components of $C_s$, generizing to irreducible components $v'$ and $w'$ of $C_{s'}$, then $\alpha(v) \leq \alpha(w)$ implies $\alpha(v') \leq \alpha(w')$.

\begin{definition}
    Let $\Rub_{g,n}$ be the subfunctor of $\Div_{g,n}$ parameterizing those pairs $(\pi:C\rightarrow S, \alpha)$ such that the values taken by $\alpha$ on the vertices of the dual graph of each geometric fiber of $C$ over $S$ are totally ordered in the partial order of $\overnorm M_S^{\rm gp}$.  We call these logarithmic divisors \emph{linearly aligned}.
\end{definition}

\begin{warning}
    Although $\Rub_{g,n}$ is a subset of $\Div_{g,n}$ when viewed as a functor on logarithmic schemes, that does not mean the underlying scheme of $\Rub_{g,n}$ is a subscheme of the underlying scheme of $\Div_{g,n}$.  Far from it:  we will see shortly that the underlying algebraic stack of $\Rub_{g,n}$ is a birational modification of the underlying algebraic stack of $\Div_{g,n}$.
\end{warning}

The following proposition gives a slightly simpler description of $\Rub_{g,n}$.  Notably, it permits us to study points of $\Rub_{g,n}$ as sections of $\pi_\ast(\overnorm M_C^{\rm gp})$ with a condition, instead of as sections of the quotient $\pi_\ast(\overnorm M_C^{\rm gp}) / \overnorm M_S^{\rm gp}$.

\begin{proposition} \label{prop:can-alpha}
Let $\Rub'_{g,n}$ be the stack of tuples $(C,\alpha)$ with ordered marked points, where $C$ is a logarithmic curve over $S$ and $\alpha$ is a global section of $\overnorm M_C^{\rm gp}$, such that, for each geometric point $s$ of $S$, the values taken by $\alpha$ on the vertices of the dual graph of $C$ are totally ordered in $\overline M_{S,s}^{\rm gp}$ and the minimal value taken is $0$.  Then the natural map $\Rub'_{g,n} \to \Rub_{g,n}$ is an isomorphism.
\end{proposition}
\begin{proof}
If $(C,\mathbf x,\alpha)$ is an object of $\Rub_{g,n}(S)$ then by definition, $\alpha$ is a section of $\pi_\ast(\overnorm M_C^{\rm gp}) / \overnorm M_S^{\rm gp}$.  But because the values of $\alpha$ taken on the vertices of the tropicalization of $C$ are totally ordered, there is a unique representative of $\alpha$ in $\Gamma(C, \overnorm M_C^{\rm gp})$ such that, in each fiber of $C$ over $S$, the smallest value it takes on the tropicalization is $0$.  
\end{proof}

\begin{theorem}\label{thm:rublogmod}
    The map $\Rub_{g,n}\rightarrow \Div_{g,n}$ is a logarithmic modification.
\end{theorem}
\begin{proof}
We can represent an $S$-point of $\Div_{g,n}$ locally as a logarithmic curve $C$ over $S$ and a piecewise linear function $\alpha$ on the tropicalization of $C$ taking values in $\overnorm M_S$.  Localizing in $S$, we can lift these elements to $M_S$ so that they define a map to the toric variety $\mathbf A^p$.  There is a minimal subdivision of the toric fan of $\mathbf A^p$ such that the $p$ coordinates are pairwise comparable.  The corresponding toric modification of $\mathbf A^p$ pulls back under $S \to \mathbf A^p$ to $S \mathop\times_{\Div_{g,n}} \Rub_{g,n}$.  Hence, $\Rub_{g,n}$ is locally the base change of a toric modification and is therefore a logarithmic modification, by definition.
\end{proof}

\setcounter{corollary}{\value{theorem}}
\begin{corollary}
The map $\Rub_{g,n} \to \Div_{g,n}$ is proper, logarithmically \'etale, birational, and a logarithmic monomorphism.
\end{corollary}
\begin{proof}
This is immediate from Theorem~\ref{thm:rublogmod} and Proposition~\ref{prop:log-mod}.
\end{proof}

\subsection{Divisions}
\label{sec:div}

Any morphism $S \to \Rub_{g,n}$ induces a divided tropical line on $S$ with at least one division.  Indeed, if $(C, \alpha) \in \Rub_{g,n}(S)$ then locally $\alpha$ can be represented by a section $\Gamma(C, \overnorm M_C^{\rm gp})$.  In fact, by Proposition~\ref{prop:can-alpha}, it is possible to choose $\alpha$ globally.
By definition of $\Rub_{g,n}$, the family of elements $\alpha(v)$, as $v$ varies among vertices of the tropicalizations of the fibers of $C$ over $S$, is totally ordered in each fiber.  Therefore the $\alpha(v)$ determine a divided tropical line $\mathcal P_\gamma$ over $S$.  Note that although the $\alpha(v)$ are only well-defined up to addition of a constant from $\overline M_S^{\rm gp}$, the divided tropical line $\mathcal P_\gamma$ is independent of this ambiguity.

The map $C \to \tropGm$ classified by $\alpha$ does not necessarily factor through $\mathcal P_\gamma$.  Figure~\ref{fig:division} shows how this can happen at the level of tropical curves, where there is no arrow from the curve in the upper right to the divided tropical line in the lower left.  However, $C \mathop\times_{\tropGm} \mathcal P_\gamma$ is a logarithmic modification of $C$ whose morphism to $\tropGm$ does factor through $\mathcal P_\gamma$, by definition.

\begin{figure}
\begin{tikzpicture}[scale=.7,arrow/.style={decoration={markings, mark=at position .5 with {\arrow{angle 90}}},postaction={decorate}}]

\def\makecoords{
\coordinate (A1) at (-3,0);
\coordinate (B1) at (-1,0);
\coordinate (C1) at (1,0);
\coordinate (D1) at (3,0);

\begin{scope}[shift={(0,2)}]
\coordinate (A2) at (-3,0);
\coordinate (B2) at (-1,0);
\coordinate (C2) at (1,0);
\coordinate (D2) at (3,0);
\end{scope}
}

\begin{scope}[shift={(-5,0)}]
\makecoords

\draw (A1) -- (D1);

\draw[black,fill=black] (B1) circle (.6mm);
\draw[black,fill=black] (C1) circle (.6mm);
\end{scope}

\begin{scope}[shift={(5,0)}]
\makecoords

\draw (A1) -- (D1);

\end{scope}

\begin{scope}[shift={(5,3)}]
\makecoords

\draw (A1) -- (D1);
\draw (A2) -- (D2);
\draw (B1) -- (C2);

\draw[black,fill=black] (B1) circle (.6mm);
\draw[black,fill=black] (C2) circle (.6mm);
\end{scope}

\begin{scope}[shift={(-5,3)}]
\makecoords

\draw (A1) -- (D1);
\draw (A2) -- (D2);
\draw (B1) -- (C2);

\draw[black,fill=black] (B1) circle (.6mm);
\draw[black,fill=black] (C2) circle (.6mm);
\draw[black,fill=black] (B2) circle (.6mm);
\draw[black,fill=black] (C1) circle (.6mm);

\end{scope}

\draw[->] (-1,0) -- (1,0);
\draw[->] (-5,2) -- (-5,1);
\draw[->] (-1,4) -- (1,4);
\draw[->] (5,2) -- (5,1);


\end{tikzpicture}

\caption{The tropicalization of a rubber morphism on the right side, with the induced subdivision of the tropical line in the lower left and the induced subdivision of the domain curve in the upper left}
\label{fig:division}
\end{figure}

\subsection{Stable maps to rubber targets}
\label{sec:rubber-maps}

\begin{definition} \label{def:rubCL}
Let $\pi : C \to S$ be family of logarithmic curves, and $L$ a line bundle on $C$.  We define $\Rub_{g,\mathbf a}(C,L)$ to be the moduli space of logarithmic trivializations (Definition~\ref{def:logtriv}) of $L$ over $C$ whose image in $\Div_{g,\mathbf a}$ lies in $\Rub_{g,\mathbf a}$.
\end{definition}

By definition, there is a cartesian diagram:
\begin{equation*} \xymatrix{
    \Rub_{g,\mathbf a}(C,L) \ar[r] \ar[d] & \Div_{g,\mathbf a}(C,L) \ar[r] \ar[d] & S \ar[d]^{(C,L)} \\
    \Rub_{g,\mathbf a} \ar[r] & \Div_{g,\mathbf a} \ar[r] & \Pic_{g,n}.
} \end{equation*}


Recall that, by Definition~\ref{def:logdiv}, an $S$-point of $\Rub_{g,\mathbf a} \subset \Div_{g,\mathbf a}$ is a logarithmic curve $C$ over $S$, a tropical line bundle $\mathcal P$ over $S$, and a map $\alpha : C \to \mathcal P$ over $S$.  The factorization through $\Rub_{g,\mathbf a}$ induces a division $\mathcal Q$ of $\mathcal P$, as in Section~\ref{sec:div}, and a logarithmic modification $\tilde C = C \mathop\times_{\mathcal P} \mathcal Q$ of $C$.

Let us suppose that $C$ is a \emph{stable} curve.  Locally on $S$, the map $\Rub_{g,\mathbf a} \to \Pic_{g,n}$ gives us a line bundle $L = \mathcal O_C(\alpha)$ on $C$, well-defined up to tensor product by a line bundle pulled back from $S$.  We write $\tilde L = \mathcal O_{\tilde C}(\alpha)$ for its pullback to $\tilde C$ and $\tilde L^\ast = \mathcal O_{\tilde C}^\ast(\alpha)$ for its associated $\Gm$-torsor.
Also locally on $S$ --- or even globally, by Proposition~\ref{prop:can-alpha} --- we can choose an isomorphism $\mathcal P \simeq \tropGm$.  There is then a $\Gm$-torsor $\logGm \to \tropGm$ and, writing $P$ for this torsor, we can identify $L^\ast = C \mathop\times_{\mathcal P} P$ and $\tilde L^\ast = \tilde C \mathop\times_{\mathcal P} P$, by Proposition~\ref{prop:univ-torsor}.  We therefore obtain a $\Gm$-equivariant map $\tilde L^\ast \to P$.
The map $\tilde C \to \mathcal P$ factors through $\mathcal Q$ by construction of $\tilde C$, so the map $\tilde L^\ast \to P$ factors through $Q = \mathcal Q \mathop\times_{\mathcal P} P$.  We now have a canonical map $\tilde L^\ast \to Q$ over $S$. 

Let $R$ be the space of $\Gm$-equivariant maps from $L^\ast$ to $Q$, viewed as a family over $C$, and let $\tilde R$ be its pullback to $\tilde C$.  Perhaps more explicitly, $R = \smash{L^\ast} ^{-1} \mathop\otimes Q$.  Then $R$ is a family of $2$-marked, genus~$0$, semistable curves over $C$ and is equipped with a section $\sigma$ over $\tilde C$, induced from the $\Gm$-equivariant map $\tilde L^\ast \to Q$ constructed above.

Taking stock, we have shown that $(C,\alpha) \in \Rub(S)$ gives rise to the following data:
\begin{enumerate}[label=\textbf{R\arabic{*}}]
\item \label{r1} a semistable logarithmic curve $\tilde C$ over $S$ with stabilization $\tau : \tilde C \to C$,
\item \label{r3} a family of $2$-marked, genus~$0$, semistable curves $R$ over $C$,
\item \label{r2} a divided tropical line $\mathcal Q$ over $S$ with $[R/\Gm] = \mathcal Q \mathop\times_S C$, 
\item \label{r4} a section $\sigma$ of $R$ over $\tilde C$ such that no component of any fiber of $\tilde C$ is carried by $\sigma$ into a node of $\mathcal Q$, and
\item \label{r5} for each component $v$ of $\mathcal Q$, the preimage in $\tilde C$ contains at least one component that is stable.
\end{enumerate}
The last condition, \ref{r5}, comes from the construction of $\mathcal Q$ and the stability of $C$:  the divisions of $\mathcal Q$ occur exactly at the images of components of $C$.  We refer to this condition as the \emph{stability} of $\sigma$.

\numberwithin{corollary}{theorem}
\begin{proposition} \label{prop:rub-stab}
Let $\Rub_{g,\mathbf a}^{\stab}$ be the locus of $(C,\alpha)$ in $\Rub_{g,\mathbf a}$ such that $C$ is stable.  Then $\Rub_{g,\mathbf a}^{\stab}$ is equivalent to the stack of quadruples $(\tilde C, \mathcal Q, R, \sigma)$ as in~\ref{r1}--\ref{r5}.
\end{proposition}
\begin{proof}
We have seen already how $\tilde C$, $\mathcal Q$, $R$, and $\sigma$ arise from $(C,\alpha)$ in $\Rub_{g,\mathbf a}$.  It remains to argue that every datum $(\tilde C, \mathcal Q, R, \sigma)$ satisfying the conditions arises uniquely by this process. 

Let $\tau : \tilde C \to C$ be the stabilization.  Since $\mathcal Q$ is a divided tropical line, it has a canonical map (in fact, two canonical maps, but we choose the convention in Proposition~\ref{prop:can-alpha}) to $\tropGm$.
The composition 
\begin{equation} \label{eqn:5}
\tilde C \to R \to \mathcal Q \to \tropGm
\end{equation}
gives us an object of $\Rub'_{g,\mathbf a}(\tilde C)$.  We wish to show it descends to $C$.

We take $Q$ to be the pullback of $\logGm$ to $\mathcal Q$ via the map $\mathcal Q \to \tropGm$ introduced above.  Since $\mathcal R = \mathcal Q \mathop\times_S C$, Theorem~\ref{thm:T} guarantees that we can recover $R$ as the space of $\Gm$-equivariant maps from a $\Gm$-torsor $L^\ast$ on $C$ to $Q$.
Consider the sheaf $R'$ of $\Gm$-equivariant maps $R \to \logGm$.  We will show that the map
\begin{equation*}
\tilde C \to R \to R'
\end{equation*}
factors through $C$.  Since the map~\eqref{eqn:5} factors through $R'$ this will give the desired factorization and complete the proof.

We show that
\begin{equation*}
R \to \tau_\ast \tau^\ast R
\end{equation*}
is an isomorphism.  This is a local assertion in $C$. Near any point $x$ of $C$, we can find an open neighborhood $U$ where $L^\ast \simeq \Gm$ and therefore where $R' \simeq \logGm$.  That is, sections of $R$ correspond to sections of $M_C^{\rm gp}$ and sections of $\tau^\ast R$ correspond to sections of $M_{\tilde C}^{\rm gp}$.  But the map
\begin{equation*}
M_C^{\rm gp} \to \tau_\ast M_{\tilde C}^{\rm gp}
\end{equation*}
is an isomorphism, because $\tau$ contracts rational components (see \cite[Appendix~B]{MR3329675} or \cite[Lemma~8.8]{CCUW}).  This completes the proof.
\end{proof}

The case where an isomorphism between $L$ and $\mathcal O_C$ is fixed is of particular interest, as this was the situation considered in~\cite{GV}.  Here the data simplify to
\begin{enumerate}[label=\textbf{R\arabic{*}$'$}]
\item \label{rp1} a semistable logarithmic curve $\tilde C$ with stabilization $\tau : \tilde C \to C$,
\item a family of $2$-marked, genus~$0$, semistable curves $Q$ over $S$, 
\item a morphism $\sigma : \tilde C \to Q$ over $S$ that does not carry any component of $\tilde C$ to a node, and
\item \label{rp4} every component of every fiber of $Q$ is covered by at least one stable component of $\tilde C$.
\end{enumerate}

\noindent As a result of the above discussion and as special cases of Proposition~\ref{prop:rub-stab} we have the following corollaries:

\setcounter{corollary}{\value{equation}}
\begin{corollary} \label{cor:rubber}
Let $\Rub(\mathcal M_{g,n}^{\log}, \mathcal O)$ be the substack of $\Rub$ parameterizing those $(C,\alpha)$ where $C$ is stable and $\mathcal O_C(\alpha)$ is pulled back from the base.  Then $\Rub(\mathcal M_{g,n}^{\log}, \mathcal O)$ is isomorphic to the moduli space of data~\ref{rp1}--\ref{rp4}.
\end{corollary}

\begin{corollary}\label{cor:rubberfromtropics}
$\Rub(\mathcal M_{g,n}^{\log}, \mathcal O)$ is isomorphic the moduli space of rubber stable maps.
\end{corollary}

\subsection{The rubber virtual fundamental class}
\label{sec:rubber-vfc}

Graber and Vakil gave a pointwise first-order deformation space and a $2$-step obstruction space for $\Rub_{g,\mathbf a}(\mathcal M_{g,n}^{\log}, \mathcal O)$ over $\Mlog_{g,n} \times \mathfrak M_{0,2}^{\rm ss}$~\cite{GV}.\footnote{Op.\ cit.\ asserts that this is an obstruction theory relative to $\frak M_{g,n}^{\log} \times \mathscr T$, but this appears to have been a misstatement.}  In this section, we will describe several global complexes that control the deformation and obstructions of $\Rub_{g,\mathbf a}(\mathcal M_{g,n}^{\log}, \mathcal \omega^{\otimes k})$ relative to different bases.  We relate our obstruction complexes to Graber and Vakil's deformation and obstruction spaces in Section~\ref{sec:gv-comparison}.

\begin{theorem} \label{thm:vfc-text}
The obstruction theories for $\Rub_{g,\mathbf a}(\mathcal M_{g,n}^{\log}, \mathcal \omega^k)$ described in Sections~\ref{sec:def}, \ref{sec:second-obs}, and \ref{sec:gv-obs} all give the same virtual fundamental class.
\end{theorem}

The proof of this theorem is given in Sections~\ref{sec:second-obs}~through~\ref{sec:gv-obs} below.

\numberwithin{theorem}{subsubsection}
\numberwithin{lemma}{subsubsection}
\setcounter{subsubsection}{\value{theorem}}
\numberwithin{equation}{subsubsection}
\subsubsection{A second obstruction theory}
\label{sec:second-obs}

Recall that $\Rub_{g,\mathbf a}(\mathcal M_{g,n}^{\log}, \omega^{\otimes k})$ is constructed from a cartesian diagram~\eqref{eqn:14} where the map $\mathcal M_{g,n}^{\log} \to \Pic_{g,n}$ sends a curve $C$ to $\omega_C^{\otimes k}$:
\begin{equation} \label{eqn:14} \vcenter{ \xymatrix{
	\Rub_{g,\mathbf a}(\mathcal M_{g,n}^{\log}, \omega_C^{\otimes k}) \ar[r] \ar[d] & \mathcal M^{\log}_{g,n} \ar[d]^{\omega^{\otimes k}} \\
	\Rub_{g,\mathbf a} \ar[r] & \Pic_{g,n}.
}} \end{equation}
Section~\ref{sec:def} gives a virtual fundamental class of $\Rub_{g,\mathbf a}(\mathcal M_{g,n}^{\log}, \omega^{\otimes k})$ by way of the canonical relative logarithmic obstruction theory associated to the morphism of logarithmically smooth algebraic stacks $\aj : \Rub_{g,\mathbf a} \to \Pic_{g,n}$.

However, the map $\mathcal M_{g,n}^{\log} \to \Pic_{g,n}$ is a locally closed embedding of smooth stacks, with normal bundle $H^1(C, \mathcal O_C)$ at a curve $C$.  Therefore there is a relative obstruction theory for $\Rub_{g,\mathbf a}(\mathcal M_{g,n}^{\log}, \omega^k)$ over $\Rub_{g,\mathbf a}$ with obstruction bundle $H^1(C, \mathcal O_C)$ giving another virtual fundamental class on $\Rub_{g,\mathbf a}(\mathcal M_{g,n}^{\log}, \omega^k)$ by Gysin pullback from $\Rub_{g,\mathbf a}$. 

We claim that these two virtual fundamental classes coincide.  In order to see this, we must recall that if $f : X \to Y$ is a morphism of logarithmic algebraic stacks with a perfect relative \emph{logarithmic} obstruction theory $E$, then $E$ gives an obstruction theory for the map $g : \undernorm X \to \Log(Y)$ induced from $f$, where $\undernorm X$ is the underlying algebraic stack of $X$.  Indeed, the relative logarithmic cotangent complex of $X \to Y$ is the relative cotangent complex of $\undernorm X \to \Log(Y)$, by definition~\cite[Definition~3.2]{Olsson-LCC}.  If $Y$ is logarithmically regular then $\undernorm Y$ is dense in $\Log(Y)$ (one can see this from the local charts~\cite[Corollary~5.25]{Olsson-Log}) so $\Log(Y)$ is equidimensional if $Y$ is, of the same dimension as $Y$.  In this case, the relative logarithmic virtual fundamental class is defined to be $g^! [\Log(Y)]$.

Since $\mathcal M^{\log}_{g,n} \to \Pic_{g,n}$ is strict, so is $\Rub_{g,\mathbf a}(\mathcal M_{g,n}^{\log}, \omega_C^{\otimes k}) \to \Rub_{g,\mathbf a}$.  We can therefore turn diagram~\eqref{eqn:14} into a cartesian diagram of algebraic stacks \eqref{eqn:21}, ignoring logarithmic structures:
\begin{equation} \label{eqn:21} \vcenter{\xymatrix{
    \undernorm{\Rub}_{g,\mathbf a}(\mathcal M_{g,n}^{\log}, \omega^k) \ar[r]^-p \ar[d]_q & \Log(\mathcal M_{g,n}^{\log}) \ar[d]^r \ar[r] & \overnorm{\mathcal M}_{g,n} \ar[d] \\
    \undernorm{\Rub}_{g,\mathbf a} \ar[r]^-s & \Log(\Pic_{g,n}) \ar[r] & \undernorm{\Pic}_{g,n}.
}} \end{equation}
The underlines denote underlying algebraic stacks (with logarithmic structure forgotten).  The square on the right is cartesian by the strictness of $\overnorm{\mathcal M}_{g,n} \to \undernorm{\Pic}_{g,n}$ and the outer square is cartesian because $\overnorm{\mathcal M}_{g,n} \to \undernorm{\Pic}_{g,n}$ is strict and~\eqref{eqn:14} is cartesian.

By \cite[Theorem~4.3]{Manolache}, we have
\begin{equation*}
q^! [\undernorm{\Rub}_{g,\mathbf a}] = q^! s^! [ \Log(\Pic_{g,n})] = p^! r^! [ \Log(\Pic_{g,n})] = p^! [\Log(\overnorm{\mathcal M}_{g,n})]
\end{equation*}
so the two virtual fundamental classes coincide. 

\subsubsection{The source of the obstruction}
\label{sec:cons-obs}

In Section~\ref{sec:gv-obs}, we will describe still another obstruction theory for $\Rub_{g,n}(\mathcal M_{g,n}^{\log}, \omega^k)$.  In order to compare that one to the obstruction theories just introduced, we will need to understand how obstructions in $H^1(C, \mathcal O_C)$ arise in practice in the obstruction theory introduced in Section~\ref{sec:second-obs}.  Consider a square-zero logarithmic lifting problem~\eqref{eqn:12} where $S'$ is a strict square-zero extension of $S$ with ideal $J$:
\begin{equation} \label{eqn:12} \vcenter{\xymatrix{
	S \ar[r] \ar[d] & \Rub_{g,\mathbf a}(\mathcal M_{g,n}^{\log}, \omega^k) \ar[r] \ar[d] & \mathcal M^{\log}_{g,n} \ar[d] \\
	S' \ar[r] \ar@{-->}[ur] \ar@{-->}[urr] & \Rub_{g,\mathbf a} \ar[r] & \Pic_{g,n}.
}} \end{equation}
It is equivalent to produce either of the dashed arrows, since the square on the right is cartesian.  The map $S \to \mathcal M^{\log}_{g,n}$ gives a logarithmic curve $C$ over $S$, and the map $S' \to \Pic_{g,n}$ gives a logarithmic curve $C'$ over $S'$ extending $C$ and an equivalence class of line bundles $L'$ on $C'$ (equivalence being tensor product by line bundles pulled back from $S'$) whose restriction to $C$ is pulled back from $S$.  If we pick an isomorphism $\psi : \omega_{C/S}^k \simeq L$, then the lifts of $\psi$ to an isomorphism $\psi' : \omega_{C'/S'}^k \simeq L'$ form a torsor under 
\begin{equation*}
\Hom(\omega_{C'/S'}^k, JL') = \Hom(\omega_{C/S}^k, J \otimes \omega^k_{C/S}) = J 
\end{equation*}
so there is an obstruction in $H^1(C, J) = J \otimes H^1(C, \mathcal O_C)$.

More specifically, the map $S' \to \Rub_{g,\mathbf a}$ gives a family of divided tropical lines $\mathcal Q'$ over $S'$ and a map $\alpha : C' \to \mathcal Q'$.  The line bundle $L'$ above, obstructing the lift, is $\mathcal O_{C'}(\alpha)$.  

We note, however, that the isomorphism $\psi$ may only exist locally on $S$.  It is therefore more canonical to choose a line bundle $K$ on $S$ and an isomorphism $\phi : \pi^\ast K \otimes \omega_{C/S}^k \to L$ on $C$.  The pair $(K, \phi)$ is uniquely determined up to unique isomorphism by $L$ (indeed, $K = \pi_\ast \Hom(\omega_{C/S}^k, L)$).  The choices of lifts of $K$ to a line bundle $K'$ on $S'$ are controlled by the complex $\mathcal O_S[1]$.  Once such a lift has been chosen, the lifts of the isomorphism $\phi$ to $\phi' : \pi^\ast K' \otimes \omega_{C'/S'}^k \to L'$ form a torsor under $\Hom(\pi^\ast K \otimes \omega_{C/S}^k, L \otimes J) = J$, by the calculation in the last paragraph, so that we have a map
\begin{equation*}
\mathcal O_S[1] \to \mathrm R \pi_\ast \mathcal O_C[1]
\end{equation*}
whose cone, $\mathrm R^1 \pi_\ast \mathcal O_C$, is the obstruction bundle of a relative obstruction theory for $\Rub_{g,\mathbf a}(\mathcal M_{g,n}^{\log}, \omega^k)$ over $\Rub_{g,\mathbf a}$.  This proves the following proposition:

\setcounter{theorem}{\value{equation}}
\begin{proposition} \label{prop:second-obs}
There is a relative obstruction theory for $\Rub_{g,\mathbf a}(\mathcal M_{g,n}^{\log}, \omega^k)$ over $\Rub_{g,\mathbf a}$ with obstruction bundle $\mathrm R^1 \pi_\ast \mathcal O_C$.
\end{proposition}

\subsubsection{A third obstruction theory}
\label{sec:gv-obs}

As explained in Section~\ref{sec:rubber-maps}, $\Rub_{g,\mathbf a}(\mathcal M_{g,n}^{\log}, \omega^{k})$ is the space of a logarithmic curve $C$, a family of $2$-marked, semistable, rational curves $Q$, and a stable (logarithmic) section of the twisted family $\omega_{C}^{-k} \otimes Q$ over $C$.  By forgetting the section, there is a \emph{logarithmic} map
\begin{equation} \label{eqn:10}
\Rub_{g,\mathbf a}(\mathcal M_{g,n}^{\log}, \omega^{k}) \to \Mlog_{g,n} \times \mathfrak M_{0,2}^{\rm ss}.
\end{equation}
In the case $k = 0$, Graber and Vakil describe a $2$-step \emph{schematic} --- that is, not logarithmic --- obstruction theory for this map.  We will also give a $2$-step obstruction theory, by thinking of the schematic morphism as a logarithmic morphism (this section) followed by a change of logarithmic structure (Section~\ref{sec:filtration}).  The filtered pieces of this obstruction theory are \emph{not} the same as Graber and Vakil's, although we will see that both theirs and ours have a common refinement.  Ours has the advantage that both steps of the filtration are perfect obstruction theories.

\setcounter{theorem}{\value{equation}}
\begin{proposition} \label{prop:gv-obs}
There is a relative logarithmic obstruction theory for $\Rub_{g,\mathbf a}(\mathcal M_{g,n}^{\log}, \omega^{k})$ over $\Mlog_{g,n} \times \mathfrak M_{0,2}^{\rm ss}$ with obstruction complex $\mathrm R \pi_\ast f^\ast T_{(\omega_{C/S}^{-k} \otimes Q) / C}[1]$.
\end{proposition}

In fact, the following lemma shows that $T_{(\omega_{C/S}^{-k} \otimes Q) / C}$ is canonically isomorphic to $\mathcal O_C$.  We've used the more complicated notation to emphasize the geometric origin of the obstructions.

\setcounter{lemma}{\value{theorem}}
\begin{lemma} \label{lem:T-triv}
Let $Q$ be a family of $2$-marked, semistable, rational curves over $S$.  Then the relative logarithmic tangent bundle $T_{Q/S}$ is canonically trivial.
\end{lemma}
\begin{proof}
One way to see this is to note that $Q$ is a division of a logarithmic line bundle $P$.  Since $P$ is a torsor under $\logGm$, its tangent bundle is the lie algebra of $\logGm$ (which is isomorphic to $\mathcal O_S$) twisted by $P$, via the adjoint action.  But $\logGm$ is commutative, so the adjoint action is trivial, and $T_{Q/S} = T_{\logGm} = \mathcal O_Q$.

A more direct argument is to note that $Q$ is locally isomorphic to $\mathbf A^1_S$ or to the pullback of the multiplication map $\mathbf A^2 \to \mathbf A^1$ along a logarithmic morphism $S \to \mathbf A^1$.  In either case, these isomorphisms are uniquely determined up to scaling by elements of $\mathcal O_S^\ast$.  By we have canonical trivializations of the duals of $T_{\mathbf A^1}$ and $T_{\mathbf A^2 / \mathbf A^1}$ given by $\frac{dx}{x}$.  These trivializations are independent of scaling by $\mathcal O_S^\ast$ and can be seen easily to glue to a global trivialization of $T_{Q/S}$.
\end{proof}

The construction of the obstruction theory in Proposition~\ref{prop:gv-obs} is standard, but in order to compare this obstruction theory to the one from Sections~\ref{sec:second-obs} and~\ref{sec:cons-obs}, it will be necessary to analyze exactly how an obstruction arises from a deformation problem~\eqref{eqn:11}:
\setcounter{equation}{\value{lemma}}
\begin{equation} \label{eqn:11} \vcenter{\xymatrix{
	S \ar[r] \ar[d] & \Rub_{g,\mathbf a}(\mathcal M_{g,n}^{\log}, \omega^{k}) \ar[d] \\	
	S' \ar[r] \ar@{-->}[ur] & \Mlog_{g,n} \times \mathfrak M_{0,2}^{\rm ss}.
}} \end{equation}
The map $S \to \Rub_{g,\mathbf a}(\mathcal M_{g,n}^{\log}, \omega^{k})$ gives a semistable logarithmic curve $C$ over $S$, a semistable family of $2$-marked, rational logarithmic curves $Q$ over $S$, and a section $f$ of $\omega_{C/S}^{-k} \otimes Q$ over $C$.  We write $\mathcal Q = [Q/\Gm]$ and $\alpha$ for the composition $C \to \omega_{C/S}^{-k} \otimes Q \to \pi^{-1} \mathcal Q$.  

The map $S' \to \Mlog_{g,n} \times \mathfrak M_{0,2}^{\rm ss}$ gives a semistable logarithmic curve $C'$ over $S'$ extending $C$ and a family of $2$-marked, semistable, rational curves $Q'$ over $S'$ extending $Q$.  Since $Q'$ is logarithmically smooth, there is no local obstruction to finding an arrow $C' \to \omega_{C'/S'}^{-k} \otimes Q'$ in diagram~\eqref{eqn:13}:
\begin{equation} \label{eqn:13} \vcenter{\xymatrix@R=10pt@C=8em{
	C \ar[r] \ar[dd] & \omega_{C'/S'}^{-k} \otimes \pi^{-1} Q' \ar[d]  \\ 
	& \pi^{-1} \mathcal Q'  \ar[r] \ar[d] & \mathcal Q' \ar[d] \\
	C' \ar[r] \ar@{-->}[ur] \ar@{-->}[uur] \ar@{-->}[urr] & C' \ar[r]^{\pi} & S'.
}} \end{equation}
Choices of lift form a torsor under $f^\ast T_{(\omega_{C/S}^{-k} \otimes Q) / C}$, and the class of this torsor in $\mathrm R^1 \pi_\ast f^\ast T_{(\omega_{C/S}^{-k} \otimes Q}$ serves as the obstruction.  This proves Proposition~\ref{prop:gv-obs}, but we analyze the obstruction a bit further to show it agrees with the obstruction from Section~\ref{sec:cons-obs}.

We note now that because $\omega_{C/S}^{-k} \otimes Q$ is a family of $2$-marked, rational, semistable curves over $S$, its logarithmic tangent bundle $T_{(\omega_{C/S}^{-k} \otimes Q) / C}$ is canonically trivial by Lemma~\ref{lem:T-triv}.  Therefore the obstruction complex here is canonically isomorphic to $\mathrm R \pi_\ast \mathcal O_C$.  In order to compare this obstruction theory to the one from Sections~\ref{sec:second-obs} and~\ref{sec:cons-obs}, we will need to show that the map
\begin{equation*}
\mathrm R \pi_\ast f^\ast T_{\omega_{C/S}^{-k} \otimes Q / C}[1] \simeq \mathrm R \pi_\ast \mathcal O_C[1] \to \mathrm R^1 \pi_\ast \mathcal O_C
\end{equation*}
carries the obstruction constructed here to the one constructed in Section~\ref{sec:cons-obs}.

Recall that if $\alpha'$ is the piecewise linear function on the dual graph of $C'$ classified by $S' \to \Rub_{g,\mathbf a}$, and $\alpha$ is its restriction to $C$, then the obstruction constructed in Section~\ref{sec:cons-obs} is the torsor of trivializations of $\omega_{C'/S'}^{-k}(\alpha')$ lifting the trivialization of $\omega_{C/S}^{-k}(\alpha)$ specified by $S \to \Rub_{g,\mathbf a}(\mathcal M_{g,n}^{\log},\omega^k)$.  


We write $\mathcal Q = [Q/\Gm]$ and $\mathcal Q' = [Q'/\Gm]$.  By definition, the obstruction constructed in this section is the torsor of lifts $C' \to \omega_{C'/S'}^{-k} \otimes Q'$ of~\eqref{eqn:13} extending the given map over $C$.  However, $\mathcal Q'$ is logarithmically \'etale over $S'$, so there is a unique extension $C' \to \mathcal Q'$ of $C \to \mathcal Q$ as in Diagram~\eqref{eqn:13}.  Therefore lifting~\eqref{eqn:13} is equivalent to lifting
\begin{equation} \label{eqn:27} \vcenter{ \xymatrix{
C \ar[r] \ar[d] & \omega_{C'/S'}^{-k} \otimes Q' \ar[d] \\
C' \ar[r] & \pi^{-1} \mathcal Q'
}} \end{equation}
Now, $Q'$ is a $\Gm$-torsor over $\mathcal Q'$.  In fact, by Proposition~\ref{prop:univ-torsor}, we know that $Q' = \mathcal O_{C'}^\ast(\alpha')$.  Therefore the torsor of lifts of~\eqref{eqn:27} is the same as the torsor of trivializations of $\omega_{C'/S'}^{-k}(\alpha')$, which (as we noted above) is the obstruction constructed in Section~\ref{sec:cons-obs}.

This completes the proof of Theorem~\ref{thm:vfc-text}.

\subsubsection{Deforming the logarithmic structure}
\label{sec:filtration}

Graber and Vakil do not work in the logarithmic category, and the map~\eqref{eqn:10} is not strict.  Therefore the first obstruction to lifting a diagram~\eqref{eqn:11} of \emph{schemes} is to promote it to a diagram in the category of \emph{logarithmic schemes}, in which $S \to S'$ is strict and $S$ has the logarithmic structure pulled back from $\Rub_{g,\mathbf a}$.

\begin{remark}
We will see that the space of promotions alluded to above is a birational modification of the base.  Since we know that the fundamental class should pull back to the fundamental class under a birational modification, we don't actually need an obstruction theory for this map.  The theory of logarithmic obstruction theories and virtual fundamental classes formalizes this.
\end{remark}

\begin{remark}
It is possible to combine the obstruction theory presented in this section with the logarithmic obstruction theory in Section~\ref{sec:gv-obs} and recognize them as filtered pieces of a $1$-step obstruction theory.  However, it seems difficult to describe the obstruction in a natural way, owing to the very different sources of the two types of obstructions (one being local to $S$ and the other being local to $C$).  Li dealt with a related issue in his moduli space of stable maps to expanded targets~\cite{Li1,Li2} (see also \cite{CMW}).

One advantage of working with logarithmic obstruction theories is that they frequently allow us to bypass such issues.
\end{remark}

To emphasize that this problem is purely about deforming logarithmic strctures, consider the factorization~\eqref{eqn:16} of~\eqref{eqn:10}:
\setcounter{equation}{\value{theorem}}
\begin{equation} \label{eqn:16}
\Rub_{g,\mathbf a}(\mathcal M_{g,n}^{\log}, \omega^k) \to \Rub_{g,n} \times \BGm \to \Mlog_{g,n} \times \mathscr T \times \BGm \simeq \Mlog_{g,n} \times \mathfrak M_{0,2}^{\rm ss}.
\end{equation}
The first map sends an $\alpha : C \to \mathcal Q$ such that $\mathcal O_C(\alpha)$ differs from $\omega_C^k$ by a line bundle pulled back from the base to $\alpha : C \to \mathcal Q$ and the line bundle that pulls back to $\omega_C^{-k}(\alpha)$.  The second map is logarithmically \'etale, so our logarithmic obstruction theories over $\Mlog_{g,n} \times \mathfrak M_{0,2}^{\rm ss}$ are also obstruction theories over $\Rub_{g,n} \times \BGm$.  We will describe a schematic obstruction theory for the second map.

Consider a \emph{schematic} lifting problem~\eqref{eqn:17}, in which $S'$ is a square-zero extension of $S$:
\begin{equation} \label{eqn:17} \vcenter{\xymatrix{
	S \ar[r] \ar[d] & \undernorm\Rub_{g,\mathbf a} \times \BGm \ar[r] \ar[d] & \undernorm\Rub_{g,\mathbf a} \ar[d] \\
	S' \ar[r] \ar@{-->}[ur] \ar@{-->}[urr] & \overnorm {\mathcal M}_{g,n} \times \undernorm{\mathscr T} \times \BGm \ar[r] & \overnorm {\mathcal M}_{g,n} \times \undernorm{\mathscr T}.
}} \end{equation}
As usual, underlines denote underlying algebraic stacks.  It is equivalent to produce either dashed arrow, so we focus on the outer rectangle.

We write $M_S$ for the logarithmic structure of $S$ pulled back from $\Rub_{g,\mathbf a}$, and $M_{S'}$ for the logarithmic structure of $S'$ pulled back from $\mathcal M^{\log}_{g,n} \times \mathscr T$.  If the dashed arrow exists, it gives a second logarithmic structure $M'_{S'}$ on $S'$ and a factorization
\begin{equation} \label{eqn:18}
(S, M_S) \to (S', M'_{S'}) \to (S', M_{S'})
\end{equation}
in which the first arrow is strict.  Conversely, any such factorization gives a lift of~\eqref{eqn:17} because $\Rub_{g,\mathbf a} \to \Mlog_{g,n} \times \mathscr T$ is logarithmically \'etale.  Thus, lifting~\eqref{eqn:17} is equivalent to finding a factorization~\eqref{eqn:18}.  

Working locally in $S$, we describe an obstruction complex controlling these lifts.  We may assume, without loss of generality, that $S$ has a geometric point such that the maps
\begin{gather*}
\Gamma(S, \overnorm M_S) \to \overnorm M_{S,s} \\
\Gamma(S', \overnorm M_{S'}) \to \overnorm M_{S',s}
\end{gather*}
are both bijections. 

Let $\mathfrak f_s : \mathfrak C_s \to \mathfrak Q_s$ be the tropicalization of the map $f_s : C_s \to Q_s$ (see \cite[Section~7]{CCUW}).  The structure of the monoids $\overnorm M_{S,s}$ and $\overnorm M_{S',s}$ can be read off from $\mathfrak f_s$.  The monoid $\overnorm M_{S,s}$ is freely generated by the lengths $\delta(e)$ as $e$ ranges among the edges of $\mathfrak C_s$ and $\mathfrak Q_s$.  The monoid $\overnorm M_{S',s}$ is the quotient of $\overnorm M_{S,s}$ in which $k \delta(e)$ is identified with $\delta(e')$ if $e$ is an edge of $\mathfrak C_s$ that maps to an edge $e'$ of $\mathfrak Q_s$ with expansion factor $k$.

The monoid of relations on $\overnorm M_{S',s}$ defining $\overnorm M_{S,s}$ is freely generated by the relations $k\delta(e) \sim \delta(e')$, described above.  To build $M'_{S'}$, we must therefore choose an identification $\mathcal O_{S'}(k\delta(e)) \simeq \mathcal O_{S'}(\delta(e'))$ for each of these relations, extending the identity $\mathcal O_S(k\delta(e)) = \mathcal O_S(\delta(e'))$.  The choices for each of these identifications form a torsor under~$\mathcal O_S$.

Once isomorphisms $\varphi_e : \mathcal O_{S'}(k\delta(e)) \to \mathcal O_{S'}(e')$ have been chosen, we have an obstruction, coming from the need for the maps
\begin{gather*}
\varepsilon_{k\delta(e)} : \mathcal O_{S'} \to \mathcal O_{S'}(k\delta(e)) \\
\varepsilon_{\delta(e')} : \mathcal O_{S'} \to \mathcal O_{S'}(\delta(e')) 
\end{gather*}
to agree, namely for the difference $\varepsilon_{\delta(e')} - \varphi_e \varepsilon_{k\delta(e)}$ to vanish.  Since this difference already vanishes over $S$, it can be regarded as a map
\begin{equation*}
\mathcal O_S \to \mathcal O_S(\delta(e')) .
\end{equation*}
We therefore obtain a tangent-obstruction complex. This is summarized by the following proposition.

\setcounter{theorem}{\value{equation}}
\numberwithin{equation}{theorem}
\begin{proposition} \label{prop:gv-obs1}
The map $\Rub_{g,n} \to \mathcal M^{\log}_{g,n} \times \mathscr T$ has a \emph{schematic} relative tangent-obstruction complex
\begin{equation} \label{eqn:28}
\prod_{e \to e'} \mathcal O_S \to \prod_{e \to e'} \mathcal O_S(\delta(e'))
\end{equation}
where the products are taken over edges of $\mathfrak C_s$ mapping to edges of $\mathfrak Q_s$.  The differential sends $\varphi_e$ in the $e \to e'$ factor to $\varepsilon_{\delta(e')} - \varphi_e \varepsilon_{k\delta(e)}$ in the same factor of the target.
\end{proposition}

\begin{remark}
In fact, $\Rub_{g,n} \to \mathcal M^{\log}_{g,n} \times \mathscr T$ is a local complete intersection morphism and~\eqref{eqn:28} is its relative tangent complex.  In the Behrend--Fantechi formalism for obstruction theories, the obstruction theory is encoded by an isomorphism from the dual of~\eqref{eqn:28} to the relative cotangent complex.
\end{remark}

\subsubsection{Comparison to Graber and Vakil's obstruction theory}
\label{sec:gv-comparison}

For simplicity, we will work on $\Rub_{g,n}(\mathcal M_{g,n}^{\log}, \mathcal O)$.  By Proposition~\ref{prop:rub-stab}, we can identify $\Rub_{g,n}(\mathcal M_{g,n}^{\log}, \mathcal O)$ with the space of stable (logarithmic) maps from logarithmic curves $C$ to $2$-marked, semistable, rational curves $Q$. 

Given such a map $f : C \to Q$, Graber and Vakil define a sheaf $f^\dagger T_Q(\log D)$ as the quotient of $f^\ast T_{\undernorm Q}(\log D)$ by its torsion subsheaf (where $D$ is the divisor at $0$ and $\infty$ of $Q$ and $T_{\undernorm Q}$ is the tangent sheaf of the underlying scheme of $Q$).

\begin{proposition} \label{prop:TQ}
There is an exact sequence
\begin{equation*} 
0 \to f^\ast T_Q \to f^\dagger T_Q \to \prod_{e \to e'} \mathcal O_S \to 0
\end{equation*}
where the product is taken over the nodes $e$ of $C$ that are carried by $f$ to nodes of $Q$.
\end{proposition}
\begin{proof}
Away from the nodes of $C$ that map to nodes of $Q$, the map $f^\ast T_Q \to f^\dagger T_Q$ is an isomorphism, so we work in a neighborhood of such a node $e$ with image $e'$.  Let $x$ and $y$ be parameters for the branches of $Q$ at $e'$ and let $v$ and $w$ be coordinates of $C$ at the branches of $e$.  After exchanging $v$ and $w$ if necessary, $x$ pulls back to $v^k$ and $y$ pulls back to $w^k$.

We can identify $T_{\undernorm Q}$ as the set of derivations from $\mathcal O_{\undernorm Q}$ into itself.  It is generated by the tangent vector fields $\delta_x$ and $\delta_y$ where
\begin{align*}
\delta_x(dx) & = x  & \delta_y(dx) & = 0 \\
\delta_x(dy) & = 0  & \delta_y(dy) & = y
\end{align*}
with the relations $y \delta_x = x \delta_y = 0$.  Then $f^\dagger T_{Q}$ is generated by $\delta_x$ and $\delta_y$ over $\mathcal O_C$, with the relations $w \delta_x = v \delta_y = 0$.

Now, $T_Q$ is freely generated by $\delta_{\log x}$, where $\delta_{\log x}(d \log x) = 1$ (and $\delta_{\log x}(d \log y) = -1$).  We note that $\delta_{\log x}(dx) = x$ and $\delta_{\log x}(dy) = -y$ so the map $T_Q \to T_{\undernorm Q}$ sends $\delta_{\log x}$ to $\delta_x - \delta_y$.

As $v \delta_x = v (\delta_x - \delta_y)$ and $w \delta_y = w (\delta_x - \delta_y)$ in $f^\dagger T_Q$, it follows that $f^\dagger T_Q / f^\ast T_Q$ is a skyscraper sheaf of length~$1$ supported at $e$ and spanned by $\delta_x \equiv \delta_y$.  On the other hand, $x$ is well-defined up to multiplication by a unit, and we argue that $\delta_x \bmod{f^\ast T_Q}$ is independent of this ambiguity.

Indeed, suppose that $u x$ is another local parameter for the same component as is $x$ (so $u$ is an invertible function).  We have $\delta_{ux}(dx) = x - u^{-1} x \delta_{ux}(du)$, so
\begin{align*}
(\delta_{u x} - \delta_x)(dx) & = - u^{-1} x \delta_{u x}(du) \\
(\delta_{u x} - \delta_x)(dy) & = 0 .
\end{align*}
There is a unique logarithmic tangent vector field $\gamma$ such that $\gamma(d \log x) = - u^{-1} \delta_{u x}(du)$.  Then
\begin{align*}
\gamma(dx) & = x \gamma(d \log x) = - u^{-1} x \delta_{u x}(du) \\
\gamma(dy) & = -y \gamma(d \log x) = u^{-1}  y \delta_{u x}(du) = 0 ,
\end{align*}
with the last equality because $\delta_{u x}(du)$ is a multiple of $x$.  Therefore $\delta_{u x} - \delta_x$ is in $f^\ast T_Q$ and the image of $\delta_x$ in $f^\dagger T_{\undernorm Q}$ is a canonical generator.
\end{proof}

We can now prove Theorem~\ref{thm:rubbervfc}.  In \cite[p.~10, Equation~(3)]{GV}, Graber and Vakil find a primary obstruction to deforming $C \to Q$ in
\begin{equation*}
H^0\bigl(C, f^{-1} \mathit{Ext}^1(\Omega_{\undernorm Q}(\log D), \mathcal O_Q) \bigr) = \prod_{e \to e'} \mathcal O_S(\delta(e')),
\end{equation*}
which is the product of the deformation spaces of the nodes $e'$ of $Q$, and therefore coincides with $\prod_{e \to e'} \mathcal O_S(\delta(e'))$, the second term of the obstruction complex in Proposition~\ref{prop:gv-obs1}.  There is a secondary obstruction\footnote{There is a typographical error in Equation (2) of loc.\ cit.: the secondary obstruction group is $H^1(C, f^\dagger T_{X_l}(-\log D_\infty))$, as in Equation~(3), and not $H^1(C, f^\ast T_{X_l}(-\log D_\infty))$.} in 
\begin{equation*}
H^1(C, f^\dagger T_Q),
\end{equation*}
which is the image of the obstruction space $H^1(C, f^\ast T_Q) = H^1(C, \mathcal O_C)$ from Proposition~\ref{prop:gv-obs}.  When these obstructions vanish, the deformations are a torsor under
\begin{equation*}
H^0(C, f^\dagger T_Q) .
\end{equation*}
Therefore the class in $K$-theory of Graber and Vakil's obstruction theory is the following virtual bundle:
\begin{equation*} 
\prod_{e \to e'} \mathcal O_S(\delta(e')) - \mathrm R \pi_\ast f^\dagger T_Q
\end{equation*}
On the other hand, the factorization
\begin{equation*}
\underline{\Rub}_{g,n}(\mathcal M_{g,n}^{\log},\mathcal O) \to \underline{\Rub}_{g,n} \to \frak M_{g,n}^{\log} \times \mathscr T
\end{equation*}
gives another $2$-step obstruction theory.  By Propositions~\ref{prop:gv-obs},~\ref{prop:gv-obs1}, and~\ref{prop:TQ} the $K$-theory class of this obstruction theory is the following sum:
\begin{align*} 
& \prod_{e \to e'} \mathcal O_S(\delta(e')) - \prod_{e \to e'} \mathcal O_S - \mathrm R \pi_\ast f^\ast T_Q  \\
= & \prod_{e \to e'} \mathcal O_S(\delta(e')) - \prod_{e \to e'} \mathcal O_S - \mathrm R \pi_\ast f^\dagger T_Q + \prod_{e \to e'} \mathcal O_S \\
= & \prod_{e \to e'} \mathcal O_S(\delta(e')) - \mathrm R \pi_\ast f^\dagger T_Q \notag
\end{align*}
This completes our comparison to Graber and Vakil's construction, and the proof of Theorem~\ref{thm:rubbervfc}.

\begin{remark}
One could consider a scheme $S$ equipped with a family of nodal curves $C$ and an invertible sheaf $L$ on $C$ such that the associated map $S \to \Pic_{g,n}$ has a perfect relative obstruction theory.  This gives a perfect relative obstruction theory for the map $\underline{\Rub}_{g,n}(C,L) \to \underline{\Rub}_{g,n}$.  Combined with Proposition~\ref{prop:gv-obs1}, this gives a $2$-step obstruction theory for the map $\underline{\Rub}_{g,n}(C,L) \to \frak M_{g,n}^{\log} \times \mathscr T$.
\end{remark}

\bibliographystyle{amsalpha}
\bibliography{rubjac}

\end{document}